\documentclass[a4paper, 11pt]{amsart}
\usepackage{graphicx}
\usepackage[english]{babel}
\usepackage{amssymb,amsmath,amsthm}
\usepackage[margin=1in]{geometry}
\usepackage{hyperref}
\usepackage{bbm}
\usepackage{booktabs}
\usepackage{tikz-cd}
\usepackage{enumitem}
\usepackage{amscd}
\usepackage{mathtools}
\setlist[enumerate]{label={(\alph*)}}
\allowdisplaybreaks

\newtheorem{theorem}{Theorem}
\numberwithin{theorem}{section}

\newtheorem{definition}[theorem]{Definition}
\newtheorem{lemma}[theorem]{Lemma}
\newtheorem{remark}[theorem]{Remark}
\newtheorem{example}[theorem]{Example}
\newtheorem{proposition}[theorem]{Proposition}
\newtheorem{corollary}[theorem]{Corollary}

\newtheorem{notation}[theorem]{Notation}

\def\K{{\mathbb K}}
\def\C{{\mathbb C}}
\def\Z{{\mathbb Z}}
\def\N{{\mathbb N}}
\def\Q{{\mathbb Q}}
\def\R{{\mathbb R}}
\def\P{{\mathbb P}}

\def\Bij{{\operatorname{Bij}}}

\DeclareMathOperator{\elem}{Elem}

\DeclareMathOperator{\Gr}{Gr}

\DeclareMathOperator{\rank}{rk}\let\rk\rank
\DeclareMathOperator{\corank}{crk}\let\crk\corank

\def\Sbundle{\mathcal S}
\def\Qbundle{\mathcal Q}
\let\delete\setminus
\def\contract{\mathbin{/}}
\DeclareMathOperator{\nullity}{nl}\let\nl\nullity
\def\tang{\mathcal T}
\def\point{\mathord{\mathrm{pt}}}
\DeclareMathOperator{\Span}{span}
\DeclareMathOperator{\conv}{conv}
\title{Equivariant Tutte polynomial}
\author[M. Bauer \and M. Doležálek \and M. Mišinová \and S. S\l obodianiuk \and J. Weigert]{
    Mario Bauer
    \and
    Matěj Doležálek
    \and
    Magdaléna Mišinová
    \and
    Semen S\l obodianiuk
    \and
    Julian Weigert
}

\begin{document}

\begin{abstract}
    We use the equivariant cohomology ring of the permutohedral variety to study matroids and their invariants. Investigating the pushforward of matroid Chern classes defined by Berget, Eur, Spink and Tseng to the product space $\P^n \times \P^n$, we establish an equivariant generalization of the Tutte polynomial of a matroid. We discuss how this polynomial encodes properties of the matroid by looking at special evaluations. We further introduce an equivariant generalization of the reduced characteristic polynomial of a matroid. 
\end{abstract}
\maketitle

\def\mate{\textit}

\section{Introduction}

Let $n\in \N$ and consider the Cremona map \begin{align*}
    \P(\C^{n+1})&\dashrightarrow \P(\C^{n+1})
    \\
    (x_0:\ldots:x_n&) \mapsto (x_0^{-1}:\ldots :x_n^{-1}).
\end{align*}
Let further $L\subseteq \C^{n+1}$ be a linear subspace representing a matroid $M$. Restricting the Cremona map to the projectivization of $L$ and taking the Zariski-closure of its graph gives a closed subvariety  $\Gamma_L\subseteq \P^n\times \P^n$. June Huh in his famous paper \cite{Huh} showed that up to sign the coefficients of the reduced characteristic polynomial of $M$ coincide with the multidegree of $\Gamma_L$. In other words he found a way to interpret a combinatorial invariant of $M$ as the class of $\Gamma_L$ in the cohomology ring of $\P^n\times \P^n$. This allowed him to use algebraic machinery, namely a generalization of the Hodge Index Theorem, to prove a longstanding conjecture about the log-concavity of the coefficients of the reduced characteristic polynomial of representable matroids and in particular of graphs. In \cite{CombinatorialHodge} this result was generalized to arbitrary matroids. 

Similarly  to \cite{Huh} utilizing an algebro-geometric interpretation of the reduced characteristic polynomial there have been several approaches to find a geometric interpretation for the full Tutte polynomial of a matroid, e.g. in \cite{tautological-classes} or in \cite{Fink-Speyer}. In our article we will focus on the approach of \cite{tautological-classes} which utilizes the permutohedral variety. This variety and its appearance in matroid theory was studied in great detail in \cite{HuhPermutohedral}.

The intersection theory of subvarieties of the permutohedral variety is a rich and interesting subject in itself. One may view this as a simpler version of the varieties of complete quadrics or complete colineations whose intersection theory allowed to answer classical questions in enumerative geometry. While this is not the main direction of our article, we point the interested reader to \cite{DMS} which contains a great overview on this topic.

Blowing-up $\P^n$ in the indeterminacy locus of the Cremona map yields the permutohedral variety $\Pi_n$,  which comes with projections to the domain and codomain of the Cremona map as the following diagram shows:
\begin{equation}\label{diagram}
\begin{tikzcd}
                        & \Pi_n \arrow[d, "\pi^n"]               &      \\
                        & \P^n \times \P^n \arrow[ld] \arrow[rd] &      \\
\P^n \arrow[rr, dashed,"\text{Crem}"] &                                        & \P^n
\end{tikzcd}
\end{equation}

It was noticed from the beginning that in this picture all objects come with a natural action of the torus $T=(\C^*)^{n+1}$, induced by the action on the domain of the Cremona map given by
\begin{equation}
    \label{eq:theaction}
    (t_0,\ldots,t_n)\cdot [x_0:\ldots:x_n] := [t_0x_0: \ldots :t_nx_n]
\end{equation} 
for $(t_0,\ldots,t_n)\in T, [x_0:\ldots: x_n] \in \P^n$. In particular, instead of working in the usual cohomology ring we can now also work in the equivariant cohomology ring. This ring is well studied for all appearing varieties and has accessible combinatorial descriptions. As an example, for $\Pi_n$ we get \begin{align*}
    H_T(\Pi_n)&=\Big\{(f_\sigma)_{\sigma \in S_{n+1}} \in \prod_{\sigma \in S_{n+1}}\Z[t_0,\ldots,t_n]\mid \forall \sigma \in S_{n+1}:\\&\forall i\in \{0,\ldots,n-1\}:f_\sigma \equiv f_{\sigma\circ (i,i+1)} \, (\text{mod } t_{\sigma(i)}-t_{\sigma(i+1)})\Big\}
\end{align*} 
where $S_{n+1}$ denotes the permutation group on $n+1$ elements $\{0,\ldots,n\}$.  
For a detailed discussion of equivariant cohomology we refer to \cite{Fulton}. \\
Each equivariant cohomology ring also comes with a natural surjection to its non-equivariant version, so we can ask for analogues of the class $[\Gamma_L]$ in the equivariant cohomology of $\Pi_n$.
In \cite{tautological-classes} the authors define tautological classes $c_i([\Sbundle_M^\vee]),c_j([\Qbundle_M])$ in  $H_T(\Pi_n)$ associated to a matroid, where $i\in \{0,\ldots, \rank(M)\},j\in \{0,\ldots,\corank(M)\}$. These are Chern classes of certain $K$-classes $[\Sbundle_M^\vee],[\Qbundle_M]$. We recall the precise definition of the Chern classes below, see Definition \ref{def_tautological_classes}. Introducing two new variables $w,z$ to keep track of the grading, we consider the classes 
\begin{align*}
        c([\Sbundle_M^\vee],z)&:=\sum_{i=0}^{\rank(M)}c_i([\Sbundle_M^\vee])z^i\in H_T(\Pi_n)[z], \\
      c([\Qbundle_M],w)&:=\sum_{i=0}^{\corank(M)}c_i([\Qbundle_M])w^i\in H_T(\Pi_n)[w]. \\  
    \end{align*}


Recall the map $\pi^n$ introduced in diagram \eqref{diagram}.
Theorem A of \cite{tautological-classes} relates the non-equivariant pushforward along $\pi^n$ of the class $c([\Sbundle_M^\vee],z)c([\Qbundle_M],w)$ to $\P^n\times \P^n$ with the Tutte polynomial of the matroid $M$. More precisely note that $\P^n\times \P^n$ has non-equivariant cohomology $H(\P^n\times \P^n)= \Z[x,y]/(x^{n+1},y^{n+1})$, so the pushforward of $c([\Sbundle_M^\vee],z)c([\Qbundle_M],w)$ to $\P^n\times \P^n$ can be understood as a polynomial with integer coefficients in four variables $x,y,z,w$ of degree at most $n$ in $x$ and $y$. This polynomial is precisely \footnote{This polynomial is equal to $(xy)^{|E|-1} t_M \left( \frac{1}{x}, \frac{1}{y}, z,w \right)$  in Theorem A from \cite{tautological-classes}. The reason for inverting the grading of the variables $x,y$ is the difference between pushing $c([\Sbundle_M^\vee],z)c([\Qbundle_M],w)$ to $\P^n\times \P^n$ as we did here and intersecting it with strict transforms of hyperplanes from the left and right $\P^n$ and then pushing the result to a point, as the authors of \cite{tautological-classes} do. For a more detailed discussion of this issue we refer to Chapter \ref{chapter 4}.}
\begin{align*}
(xy)^{|E|-1}\left(\frac{1}{x}+\frac{1}{y}\right)^{-1}\left(\frac{1}{y}+z\right)^{\rank(M)}\left(\frac{1}{x}+w\right)^{|E|-\rank(M)}T_M\left(\frac{x + y}{x(yz+1)}, \frac{x + y}{y(wx+1)}\right)
\end{align*}
where $T_M(x,y):=\sum_{S\subseteq E}(x-1)^{\rank(M)-\rank_M(S)}(y-1)^{|S|-\rank_M(S)}$ denotes the Tutte-polynomial of $M$.
This can be viewed as a generalization of the connection between $[\Gamma_L]$ and the characteristic polynomial. In fact for representable matroids, $[\Gamma_L]$ coincides with the non-equivariant top Chern class $c_{\corank(M)}([\Qbundle_M])$.

In this article we will generalize this formula to equivariant cohomology, that is we compute the pushforward of $c([\Sbundle_M^\vee],z)c([\Qbundle_M],w)$ to $\P^n\times \P^n$ as an element of $H_T(\P^n\times \P^n)$. The equivariant cohomology ring of $\P^n\times \P^n$ under the $T$-action \eqref{eq:theaction} is isomorphic to \[\frac{\Z[t_0,\ldots,t_n][x,y]}{\prod_{i=0}^n(x+t_i),\prod_{i=0}^n(y-t_i)}.\] The map $H_T(\P^n\times \P^n)\to H(\P^n\times \P^n)$ which forgets the torus action is in this description given by substituting $t_i=0$ for all $i=0,\ldots,n$. Therefore in fact our computation will specify to Theorem A of \cite{tautological-classes} precisely when setting $t_i=0$ for all $i=0,\ldots,n$. 
Since the non-equivariant version relates the pushforward $\pi^n_*(c([\Sbundle_M^\vee],z)c([\Qbundle_M],w))$ to the Tutte polynomial $T_M$ of $M$, we can ask what plays the role of $T_M$ in the equivariant setting. Following the suggestion in \cite[§4.5]{Mateusz-paper}, this leads us to the definition of the equivariant Tutte polynomial.

In Section \ref{Preliminaries}
we will recall some standard definitions and results that we will use throughout the article. Most importantly we will briefly recall the definition of the Permutohedral variety and its equivariant cohomology ring as well as the tautological classes associated to a matroid from \cite{tautological-classes}.
We point the interested reader to \cite{HuhPermutohedral} for more details on the intersection theory of the permutohedral variety and its combinatorial properties.
For more background on matroids in general we refer to \cite{matroid theory}, for equivariant cohomology our main source is \cite{Fulton}. A very good introduction to toric varieties can be found in \cite{cls}, for a quick overview we recommend \cite[Chapter 8]{Mateusz-Bernd-Book}.

Section \ref{sec:push} contains the main computation. Our main tool to do so is recursion on matroids: All appearing cohomology classes associated to a matroid $M$ can be described in terms of the corresponding classes associated to the contraction $M\contract e$ and deletion $M\delete e$, where $e\in E$. This allows us to inductively prove Theorem \ref{thrm:pushforward-graded-chern} which gives an explicit formula for $\pi_*^n(c([\Sbundle_M^\vee],z)c([\Qbundle_M],w))$ in terms of the rank function of the matroid $M$. \\
We then proceed to motivate the definition of the equivariant Tutte polynomial in Section \ref{chapter 4} as follows.
\begin{definition}[Equivariant Tutte polynomial]
    Let $M$ be a matroid with ground set $E$. Then we define the \emph{equivariant Tutte polynomial} of $M$ as the following four-variable polynomial with coefficients in $\Z[E]$:
    $$
        \widehat{T}_M(x,y,r,s) = \sum_{S\subseteq E} (x-1)^{\rank(M)-\rank_M(S)}(y-1)^{\nullity_M(S)}\prod_{e\in S}(1+r t_e) \prod_{e\notin S} (1+s t_e).
    $$
    By convention, we set $ \widehat{T}_M(x,y,r,s)=1\in \Z$ for $M$ supported on the empty ground set.
\end{definition}
The equivariant Tutte polynomial and the \emph{multivariate Tutte polynomial} from \cite{PottsSokal} both recover the matroid completely and hence can be obtained from each other, see \eqref{potts}. Our computations can be understood as bringing the multivariate Tutte polynomial into the algebro-geometric world.

The remainder of Section \ref{chapter 4} is dedicated to checking 
basic properties of $\widehat{T}_M(x,y,r,s)$ such as recursive behavior, compatibility with direct sums and duals of matroids and valuativity. In particular we prove that  the multivariate Tutte polynomial of \cite{PottsSokal} is a valuative function.

One reason for the popularity of the usual Tutte polynomial is that this invariant can specify to any other generalized Tutte-Grothendieck invariant, that is any invariant behaving nicely with respect to deletion and contraction of elements of the matroid. In Section \ref{section:equi-T_G} we will show that the equivariant Tutte polynomial satisfies a similar universality result among functions assigning some value in a ring to labeled matroids and behaving well with respect to deletion and contraction. \\
In the final chapter we are interested in the combinatorial properties of the equivariant Tutte polynomial. One of the key results is Proposition \ref{evaluating_r_s} which states that setting the last two variables $(r,s)$ to $(1,0)$ (respectively $(0,1)$) read as a polynomial in the variables $t_e$ has coefficients in $\Z[x,y]$ that are usual Tutte polynomials of smaller matroids obtained from $M$ via contraction (respectively deletion) of subsets. This allows us to identify a number of combinatorial interpretations for certain evaluations of the equivariant Tutte polynomial, they are summarized in table \ref{tab:eval_of_tutte}. We also define an equivariant analogue of the reduced characteristic polynomial of a matroid, which for graphic matroids is closely related to the chromatic polynomial. Just like for the Tutte polynomial this is done by performing the pushforward of $c([\Qbundle_M])$ to $\P^n\times \P^n$ equivariantly. The result is the following definition.
\begin{definition}
    Let $M$ be a matroid on a groudset $E$. We define the equivariant reduced characteristic polynomial $\widehat{\chi}_M(q)\in \Z[E][q]$ of $M$ as follows:
    $$
        \widehat{\chi}_M(q) = \frac{(-1)^{\rank(M)}}{q-1}\cdot \widehat{T}_M(1-q, 0, q, 1) = \frac{1}{q-1} \cdot \sum_{S\subset E} q^{\rank(M)-\rank_M(S)}(-1)^{|S|}\cdot \prod_{e\notin S} (1+t_e) \prod_{e\in S} (1+qt_e).
    $$
\end{definition}
We observe that both $\widehat{T}_M$ and $\widehat{\chi}_M$ as well as many of their specializations recover the labelled matroid $M$. Our final result investigates precisely which evaluations of $\widehat{T}_M$ can still recover $M$.
\bigskip \\
\textbf{Acknowledgements.} We thank Mateusz Micha\l ek who introduced us to this subject and supervised us in the process of writing this article. We further thank Andrzej Weber 
 and Jaros\l aw A. Wiśniewski for helpful discussions on equivariant cohomology rings. Lastly we thank the reviewers for their careful reading and for pointing out missing references. The second, third and fourth author are partially funded by the German Academic Exchange Service DAAD.  

\section{Preliminaries}
\label{Preliminaries}
We start by briefly recalling a useful tool for studying equivariant cohomology: the localization principle. We further describe how the equivariant Gysin homomorphism looks like.
In Section \ref{polyt} we will describe a polytope called permutohedron, to which we associate a smooth, toric variety in Section \ref{variet}, called the permutohedral variety. In Section \ref{coho} we will briefly discuss its equivariant cohomology ring and how the localization principle applies to it. In Section \ref{pxp} we will describe how to pass between localization formula and another description of equivariant cohomology. We also introduce an equivariant map $\pi^n:\Pi_n\to \P^n\times \P^n$ which will play a crucial role in Section \ref{sec:push}. In the last section of the preliminaries we recall the construction of tautological classes of matroids which were introduced in \cite{tautological-classes}. The pushforward along $\pi^n$ of these classes gives rise to many interesting polynomials, for example, an equivariant version of the Tutte polynomial, which motivates us to write this article. 

\subsection{Localization and Gysin pushforward}

For the discussion which follows we make a technical definition:
\begin{definition} \label{coprime}
We say that characters $\chi_1,\chi_2$ of a torus $T$ are \emph{relatively prime} if there are no $a,b\in \Z\setminus \{0\}$ such that $\chi_1^a=\chi_2^b$.
\end{definition}
We recall a version of the localization principle which was proven in \cite{Chang-S} and \cite{modernCS}. For a modern exposition we also refer to \cite[Corollary 7.4.3]{Fulton}: 

\begin{theorem}\label{localization}
Let $X$ be a nonsingular variety with an action of a torus $T$ such that the set of torus fixed points $X^T$ is finite. Let $H_T(X)$ be the equivariant cohomology ring with coefficients in $\Z$.
Assume that $H_T(X)$ is free over $H_T(\point)$. Suppose that for each $p\in X^T$, the characters acting on the tangent space over $p$ are relatively prime.
Then the ring $H_T(X^T)$ is just $\Pi_{p \in X^T} H_T(\point)$, and $H_T(X)\subset H_T(X^T)$ consists of exactly those tuples $(f_p)_{p \in X^T} \in H_T(X^T)$
 where the difference $f_p-f_q$ is divisible by the character acting on $1$-dimensional orbit $c_{p,q}$ connecting points $p,q\in X^T$ for all such $c_{p,q}$.
\end{theorem}

Suppose that we have smooth varieties $X,Y$ equipped with action of a torus $T$, such that the torus-fixed loci $X^T$ and $Y^T$ are finite. Fix an equivarant map $f : X \to Y$. Let $\iota_X:X^T\to X$ and $\iota_Y:Y^T\to Y$ be the inclusions of the fixed loci. Then the pullback map $f^*:H_T(Y)\to H_T(X)$ of a class $c\in H_T(Y)$ viewed from the perspective of localization principle looks as follows:
\[ (f \circ \iota_X)^*(c) = \big(\iota_Y^*(c)_{f(p)}\big)_{p \in X^T}.  \]

For every $p\in X^T$ and $q\in Y^T$
denote by $\tang^X_p$ and $\tang^Y_q$ the product of the characters acting on the tangent spaces over $p$ and over $q$ respectively.
The following proposition describes how the equivariant pushforward in cohomology along the map $f$ looks like in terms of localization.

\begin{proposition}\label{gysin}
For a class $c\in H_T(X)$,  points $p\xhookrightarrow{\iota_p} X^T$ and $f(p)\xhookrightarrow{ \iota_{f(p)}} Y^T$ the following diagram is commutative:
\[
\begin{tikzcd}
H_T(X) \arrow{r}{f_*}\arrow{d}{\iota_p^*}  & H_T(Y) \arrow{d}{\iota_{f(p)}^*}
\\%
 H_T(p)  \arrow{r}{ \cdot \frac{\tang^Y_{f(p)} }{\tang^X_p}}  &  H_T(f(p)).
\end{tikzcd}
\]
Summing over all points $p\in f^{-1}(q)$ gives a formula\[\iota_Y^*(f_*(c))_q = \sum_{\substack{p\in f^{-1}(q)}} \frac{\tang^Y_{q} }{\tang^X_p}\cdot \iota_X^*(c)_p. \]In particular, if $q \not \in f(X)$ then $\iota_Y^*(f_*(c))_q = 0$.
\end{proposition}

\subsection{The permutohedron} \label{polyt}
Let $S_{n+1}$ be the group of all bijections of the set $[n]=\{0,1,\dots,n\}$ onto itself. For every $\sigma\in S_{n+1}$ we define the corresponding point $p_\sigma$ to be
\[p_\sigma := (\sigma^{-1}(0),\sigma^{-1}(1),\dots,\sigma^{-1}(n)) \in \R^{n+1}.\]
The \emph{permutohedron} $P_n$ is 
defined as the convex hull of the $p_\sigma$. One can easily show that it has nonempty interior in the affine hyperplane $\{v \in \R^{n+1} : \sum_{i=0}^n v_i = \frac{n (n+1)}{2}\}$, hence it is an $n$-dimensional polytope, and its vertices are exactly the $p_\sigma$. 
For every nonempty, proper subset $S \subsetneq [n]$ we have a corresponding facet $F_S$ of $P_n$, where
\[  F_S = \conv(\{v \in P_n \mid v_i \in \{0, \ldots, |S|-1\} \text{ for all } i \in S\}   ).  \]
Consider the linear functional $\ell_S : \R^{n+1} \to \R, v \mapsto \sum_{i \in S} v_i$. Since every $v \in P_n$ clearly satisfies $\ell_S(v) \geq 0 + 1 + \ldots + (|S|-1) =: c_S$, $P_n$ is contained in the closed halfspace $\{v \in \R^n : \ell_S(v) \geq c_s\}$ and the intersection of $P_n$ with the affine hyperplane $\{v \in \R^n : \ell_S(v) = c_S\}$ is exactly $F_S$.
In fact, every facet of $P_n$ arises in this way. This also determines all of the faces of $P_n$, since every face is the intersection of the facets containing it. Two vertices $p_{\sigma}, p_{\sigma'}$ of $P_n$ are connected by an edge, if and only if there is some $i \in \{0, \ldots, n-1\}$ such that $\sigma = \sigma ' \circ (i, i+1).$

\subsection{The permutohedral variety}\label{variet}
Let us consider an action of the $(n+1)$-dimensional torus $T\cong (\C^*)^{n+1}$ on a vector space $V\cong \C^{n+1}$ given in coordinates by $$(t_0,t_1,\dots,t_n)\cdot (x_0,x_1,\dots,x_n)=(t_0x_0,t_1x_1,\dots,t_nx_n)$$ for every $(t_i)_i\in T$ and $(x_i)_i\in V$. The above action induces a natural action of $T$ on $\bigwedge\nolimits^k \C^{n+1}$ as well as on $\P(\bigwedge\nolimits^k \C^{n+1})\cong \P^{\binom{n+1}{k}-1}$, where $k$ is any positive integer.
Let $\binom{[n]}{k}$ denote the set of all $k$-element subsets of $[n]$. If $e_0, \ldots, e_n$ is the standard basis of $\C^{n+1}$, the vectors $(\bigwedge_{i \in S} e_i)_{S  \in \binom{[n]}{k}}$ form a basis of $\bigwedge^k \C^{n+1}$, which is why we use coordinates $p = [p_S]_{S \in \binom{[n]}{k}}$ or often just $p = [p_S]_{S}$ for a point $p \in \P^{\binom{n+1}{k} - 1}$.

For all $k = 1, \ldots, n$ we have an embedding
\begin{equation}\label{phi_i}
\phi_k:T \to \P^{\binom{n+1}{k}-1}, (t_i)_i \mapsto \big[\prod_{i \in S} t_i\big]_S.    
\end{equation}

Let $I_k$ be the ideal of $\C[x_0,\dots,x_n]$ generated by all monomials of degree $k$. We define the map 
\begin{equation}    \label{phi}
\phi:T \to \prod_{k=1}^n \P^{\binom{n+1}{k}-1}
\end{equation}
as the product of the maps $\phi_k$.
The $n$-dimensional \emph{permutohedral variety} $\Pi_n$ is defined as the closure of the graph of the map $\phi$. Since the graph of $\phi$ is constructible, Zariski- and Euclidean closure coincide.
Using the Segre embedding $\psi:\prod_{k=1}^n \P^{\binom{n+1}{k}-1}\to \P^{\prod_{k=1}^n\binom{n+1}{k}-1}$ one can also view the above construction as a blowup of $\P^n$ 
along the closed scheme given by the homogeneous ideal $I:=\prod_{k=2}^n I_k$. One can also realize this blowup as a sequence of consecutive blowups: starting from $I_2$ and ending at $I_n$, which corresponds to firstly blowing up the points that have all but one coordinate nonzero, secondly blowing up the strict transforms of the projective lines given by the vanishing all but two coordinates, and so on. Note that at every step the strict transforms will be disjoint, because the previous blowup was along an intersection. In particular every consecutive blowup will be smooth and toric, since we always blow up along smooth, torus invariant loci. It can be shown that $\Pi_n$ is a resolution of singularities of the graph of the \emph{Cremona transformation}
\[
\P(\C^{n+1})\dashrightarrow \P(\bigwedge\nolimits^n \C^{n+1}),
\]
which simply inverts homogeneous coordinates. Note that this rational map is $T$-equivariant.
If we consider the map 
\begin{equation}\label{big}
\psi \circ \phi : T \to \P^{\prod_{k=1}^n\binom{n+1}{k}-1},
\end{equation}
one can easily see that the permutohedron $P_n$ is the polytope corresponding to the permutohedral variety, since $P_n$ is the Minkowski-sum of the hypersimplices $\Delta_k := \conv(\sum_{i \in S} e_i : S \subseteq [n], |S| = k)$, and the Minkowski-sum of polytopes corresponds exactly to the Segre-embedding of varieties.

\subsection{Torus orbits and the Cohomology Ring of $\Pi_n$}\label{coho}

So far we have worked with $n+1$ dimensional torus $T \cong (\C^*)^{n+1}$ acting on $\Pi_n$. From the perspective of toric varieties, it is also natural to consider the embedded torus $i:T'\to \P(\C^{n+1})$ which is simply the quotient of $T$ by the one-dimensional isotropy subtorus scaling all coordinates of $\C^{n+1}$ simultaneously.   
Note that the map $p:T\to T'$ is surjective. In particular orbits of the action of $T$ are orbits of the action of $T'$ and we can use toric geometry to easily describe the latter.
For any projective toric variety that is given by a polytope, the torus fixed points correspond to the vertices, and one-dimensional orbits correspond to the edges of the polytope. It is thus convenient to index fixed points of $\Pi_n$ by permutations.

\begin{remark}\label{limit}
A torus-fixed point $q_\sigma$ corresponding to a vertex $p_\sigma$ can be obtained as a limit of a special monomial curve $c^\sigma$ defined on $T'$. Firstly, consider the curve 
\[c_1^\sigma:\C^*\to T, \ t\to(t^{\sigma^{-1}(0)},t^{\sigma^{-1}(1)},\dots,t^{\sigma^{-1}(n)})).\]
Composing with the map $i\circ p$ yields a curve $c^\sigma$ in the permutohedral variety $\Pi_n$. Identifying $\Pi_n$ with the image of its embedding coming from the polytope $P_n$ (cf. \eqref{big}), we see that $\lim_{|t| \to \infty} c^\sigma(t) = [0:\dots:0:1:0:\dots:0] = q_\sigma$, where the only nonzero coordinate is the one indexed by the vertex $p_\sigma$.
\end{remark}

In a similar manner one can show that the one-dimensional orbits are tori given by the nonvanishing of coordinates $p_\sigma$ and $p_{\sigma\circ (i,i+1)}$ and compute that $T$ acts on this orbit with the character $t_{\sigma(i)}\cdot t^{-1}_{\sigma(i+1)}$.
Combining the above data with Theorem \ref{localization} we arrive at the following result, which is Theorem 2.1 in \cite{tautological-classes}:
\begin{proposition}\label{H(Pi)}
The equivariant cohomology ring $H_T(\Pi_n)$ of $\Pi_n$ can be viewed as a subring of \[ H_T(\Pi_n^T)=\prod_{\sigma\in S_{n+1}}\Z[t_0,t_1,\dots,t_n]\] via restriction to the fixed points. An element $(f_\sigma)_{\sigma \in S_{n+1}} \in H_T(\Pi_n^T)$ belongs to the image of this map, if and only if $f_\sigma-f_{\sigma\circ (i,i+1)}=0 \pmod{\ t_{\sigma(i)}-t_{\sigma(i+1)}}$ holds for all $\sigma \in S_{n+1}$ and $i \in \{0, \ldots, n\}$.
\end{proposition}

\subsection{Equivariant cohomology of $\P^n\times \P^n$ and the map $\pi^n:\Pi_n\to \P^n\times \P^n$} \label{pxp}

Let us consider again the $(n+1)$-dimensional torus $T = (\C^*)^{n+1}$ that acts on $\C^{n+1}$ by scaling each coordinate, as in Section \ref{variet}. The space $\bigwedge^n \C^{n+1}$ can 
be identified with the dual vector space of $\C^{n+1}$. We thus fix the convention that $t = (t_0, \ldots, t_n) \in T$ acts on $\bigwedge^n \C^{n+1}$ by scaling the $i$-th basis vector $\bigwedge_{j \neq i} e_j$ with $t_i^{-1}$, where the $(e_1,\dots,e_n)$ is the standard basis of $\C^{n+1}$.
Proposition $6.1$ from \cite[§2]{Fulton} allows us to describe the equivariant cohomology ring of $\P(\C^{n+1})$: 
\[
H_{T}(\P(\C^{n+1}))=\frac{\Z[t_0,t_1,\dots,t_n][\alpha]}{\big(\prod_{i=0}^n (\alpha+t_i)\big)}.
\]
Here the variables $t_i$ correspond to the equivariant first Chern class of the trivial bundle over $\P(\C^{n+1})$, with action given by rescaling the $i$-th coordinate by $t_i$,
and $\alpha$ is the equivariant first Chern class of the dual of the equivariant tautological bundle. Similarly, because we fixed the convention of the action of $T$ on $\bigwedge \nolimits^n \C^{n+1}$, we get the natural identification 
 \[
 H_T(\P(\bigwedge\nolimits^n \C^{n+1}))=\frac{\Z[t_0,t_1,\dots,t_n][\beta]}{\big(\prod_{i=0}^n (\beta-t_i)\big)}.
 \]  
Here $\beta$ is again the equivariant first Chern class of the dual tautological bundle over $\P(\bigwedge\nolimits^n\C^{n+1})$.
The product $\P(\C^{n+1})\times \P(\bigwedge\nolimits^n \C^{n+1})$ is one of the main varieties we will discuss in this article. By abuse of notation we will often just write $\P^n \times \P^n$. The Kunneth formula yields that
\begin{equation}
\label{eq:middle-guy-cohomology}
H_{T}\left(\P^n\times \P^n) \right)=\frac{\Z[t_0,t_1,\dots,t_n][\alpha,\beta]}{\big(\prod_{i=0}^n (\alpha+t_i),\prod_{i=0}^n(\beta-t_i)\big)}.
\end{equation}
Every element of the equivariant cohomology ring can be identified with a polynomial $F$ in $\Z[t_0,t_1,\dots,t_n][\alpha,\beta]$ of degree less than $n+1$ with respect to both $\alpha$ and $\beta$.

 Note that for every two torus fixed points $a\in \P(\C^{n+1})$ and $b\in \P(\bigwedge\nolimits^n \C^{n+1})$ we have a fixed point $(a,b)\in\P(\C^{n+1})\times \P(\bigwedge\nolimits^n  \C^{n+1})$ in the product, and every fixed point arises in such way. 
We now show how to pass from the description \eqref{eq:middle-guy-cohomology} to the description via localization.
\begin{proposition}
The restriction map $H_T(\P^n\times \P^n)\to H_T([\P^n\times \P^n]^T)$ is given by
\begin{align*}
    \frac{\Z[t_0,t_1,\dots,t_n][\alpha,\beta]}{\left(\prod_{i=0}^n (\alpha+t_i),\prod_{i=0}^n(\beta-t_i)\right)}\to \prod_{0\leq i,j\leq n} \Z [t_0,\dots,t_n],
\end{align*}
where each $t_k$ is mapped to $t_k$, $\alpha$ is mapped to $(-t_i)_{i,j}$ and $\beta$ is mapped to $(t_j)_{i,j}$.
\end{proposition}
{\parindent0mm Indeed, the torus acts on the product of tautological bundles $\mathcal{O}_{\P(\C^{n+1})}(-1)\times \mathcal{O}_{\P(\bigwedge\nolimits^n  \C^{n+1})}(-1)$ over fixed point $(i,j)$ with characters $t_i$ and $-t_j$ respectively.}

For $n>1$ we define 
\begin{equation} \label{MAP}
\pi^n:\Pi_n\to \P(\C^{n+1})\times \P(\bigwedge\nolimits^n  \C^{n+1})
\end{equation}
to be the composition of the closed embedding $\Pi_n\to \prod_{k=1}^n\P(\bigwedge\nolimits^k  \C^{n+1}))$ with the natural projection to the first and the last factors. For $n=1$ we see that $\Pi_1\cong \P^1$ and we define $\pi^1:\Pi_1\to \P^1\times \P^1$ to be the identity on the first factor and the Cremona transformation on the second (note that the Cremona transformation is an automorphism of $\P^1$ and in particular a regular map).
In the case $n=0$ we define $\Pi_0$ to be a point and formally $\pi^0:\Pi_0\to \P^0\times \P^0$ is the identity.

\begin{remark}\label{indexing}
$T$-fixed points of $\Pi_n$ are mapped to $T$-fixed points of $\P^n\times \P^n$.
While the fixed points of $\Pi_n$ are indexed by permutations, the fixed points of $\P^n\times \P^n$ can be indexed by $[n]\times [n]$ in a natural way.
Note that the image of $\phi,\phi_1$ and $\phi_n$ is dense in respectively $\Pi_n,\P(\C^{n+1})$ and $\P(\bigwedge\nolimits^n \C^{n+1})$. In particular, fixed points can be described as limits of monomial curves as in Remark \ref{limit}. 
Composing a curve $c_1^\sigma$ (as in Remark \ref{limit}) with $\pi^n\circ \phi = \phi_1\times \phi_n$, and looking at the coordinate with the largest exponent, we see that the torus-fixed point $q_\sigma$ gets mapped to $(\sigma(n),\sigma(0))$ by $\pi^n$.
\end{remark}

\subsection{Matroids and their tautological (Chern) classes}

A \emph{matroid} $M$ is a pair $M=(E,B)$, where $E$ is a finite set and $B$ is a non-empty collection of subsets of $E$, called the \emph{bases} of $M$. The bases of $M$ are required to satisfy the \emph{base exchange property}, which is inspired by the Steinitz exchange lemma for vector spaces: If $B_1, B_2 \in B$ are two bases and if $v \in B_1 \setminus B_2$, then there exists some $w \in B_2 \setminus B_1$ such that $(B_1 \setminus \{v\}) \cup \{w\} \in B$ is a basis. Matroids generalize the concept of linear dependence in a vector space. For a comprehensive introduction to matroid theory we recommend \cite{matroid theory}.

When a subset $S\subseteq E$ is considered, we denote its \emph{rank} as $\rank_M(S)$ and its \emph{nullity} as $\nl_M(S) := |S|-\rank_M(S)$. We also write $\rank(M):=\rank_M(E)$ and $\corank(M):=\nullity_M(E)$ for the rank of $M$ and its dual respectively. We further say that an element $e\in E$ is $\emph{general}$ (in $M$), if it is neither a loop nor a coloop in $M$. By abuse of notation, we will sometimes write $\emptyset$ for the unique matroid with groundset $E=\emptyset$.

For this article the most important invariant of a matroid will be its \emph{Tutte  polynomial} $T_M(x,y)$, which was defined in \cite{tutte and applications} as
\[
T_M(x,y)=\sum_{S\subset E} (x-1)^{\rk(M)-\rank_M(S)}\cdot (y-1)^{\nl_M(S)}\in \Z[x,y].
\] This invariant satisfies the following deletion-contraction relation for every matroid $M=(E,B)$ and each element $e\in E$:
\[
T_M(x,y) = \begin{cases}
            T_{M\delete e}(x,y)+T_{M\contract e}(x,y), & \text{if $ e$ is a general element in $M$,}\\
            yT_{M\delete e}(x,y),& \text{if $e$ is a loop in $M$,}\\
            xT_{M\contract e}(x,y),& \text{if $e$ is a coloop in $M$.}
            \end{cases}
\]
Together with the base case $T_\emptyset(x,y)=1$ this also uniquely determines $T_M(x,y)$ for every matroid.\\

In \cite{tautological-classes}, the authors assign classes in the equivariant cohomology ring of the permutohedral variety to a matroid $M$ with groundset $[n]:=\{0,\ldots,n\}$. To do so they consider the \emph{lex-first-basis} of $M$ associated to a permutation $\sigma\in S_{n+1}$, which can be constructed as follows.
First order the elements of the groundset according to $\sigma$:
\[
(\sigma(0),\sigma(1),\dots,\sigma(n)).
\]
Set $I_{-1}=\emptyset$ and traverse the above list from left to right in $n+1$ steps. For $k=0,\ldots,n$ set $I_k:=I_{k-1}\cup \{\sigma(k)\}$ if $I_{k-1}\cup \{\sigma(k)\}$ is independent and set $I_k:=I_{k-1}$ otherwise. Then $I_{n}$ will be a basis of $M$ and varying $\sigma$ every basis of $M$ will arise in this way. Equivalently $I_{n}$ is the first basis that appears if one orders all bases of $M$ lexicographically according to $\sigma$.

\begin{definition}[lex-first-basis]\label{lex} 
    The basis $I_{n}$ in the above is called the $\emph{lex-first-basis}$ of $M$ associated to the permutation $\sigma$ and is denoted by $B_\sigma$. 
\end{definition}
    We illustrate the construction lex-first-bases on the following example.
\begin{example}
 Let $M$ be the matroid on the groundset $E=\{0,1,2,3\}$ where every subset of $E$ is independent except for the circuit $\{0,1,2\}$ and the whole set $E$. This matroid can be represented by the vectors $\{e_0,e_1,e_0+e_1,e_2\}\subseteq \C^3$.
  For $\sigma=\text{id}$ being the trivial permutation we get $I_0=\{0\}$, $I_1=\{0,1\}=I_2$ and $I_3=\{0,1,3\}$. On the other hand for a permutation $\sigma$ defined by 
\[
(\sigma(0),\sigma(1),\sigma(2),\sigma(3))=(2,1,0,3)
\]
we get that $I_0=\{2\}$, $I_1=\{1,2\}=I_2$ and $I_3=\{1,2,3\}$.

\end{example}

For any finite set $E$, we define
\[
    \Z[E] := \Z[t_e \mid e\in E]
\]
to be the polynomial ring over $\Z$ with free variables indexed elements of $E$. For $E'\subseteq E$ we will consider $\Z[E']$ as a subring of $\Z[E]$.

\begin{definition}[tautological classes of matroids]
\label{def_tautological_classes}
Let us fix a matroid $M$ on ground set $E=[n]$.  We define the \emph{$i$-th  tautological sub Chern class} of the matroid $M$ as $c_i([\Sbundle_M])\in \prod_{\sigma\in S_{n+1}}\Z[E]$ with
\begin{equation}\label{sbundle}
c_i([\Sbundle_M])_\sigma := 
\elem_i(\{-t_e\}_{e\in B_\sigma}),
\end{equation}
where $\elem_j$ stands for the $j$-th elementary symmetric polynomial.

Similarly, we define the \emph{$i$-th tautological quotient Chern class} $c_i([\Qbundle_M])$ through
\begin{equation}\label{qbundle}
c_i([\Qbundle_M])_\sigma:=  \elem_i(\{-t_e\}_{e\notin B_\sigma}).
\end{equation}
\end{definition}
Using Theorem \ref{localization} one can check that the class $c_i([\Sbundle_M])$ lies in the image of the restriction $H_T(\Pi_n)\to H_T(\Pi_n^T)$. In Section~\ref{subsec:representable}, we supply a concrete geometrical meaning of the tautological classes for representable matroids: this is the first equivariant Chern class of the bundle $\Sbundle$ in the Lemma \ref{subbundle}.

We will also work with $i$-th Chern classes of the duals of the above which are defined simply by changing $-t_e$ to $t_e$ in all polynomials and geometrically indeed corresponds to taking the dual. For example:

\[
c_i([\Sbundle_M^\vee]):=\elem_i(\{t_e\}_{e\in B_\sigma})_\sigma.
\]
\begin{remark}
    The experienced reader might suspect from our notation the existence of $K$-classes $[\Sbundle_M],[\Qbundle_M]$. These classes are constructed in \cite{tautological-classes} and are represented by vector bundles when $M$ is realizable over $\C$. Since in the present article we only deal with the different Chern classes, which can be understood combinatorially, we refrain from mentioning the $K$-theoretic side. We also point out the article \cite{GKM24}, which is similar in spirit to our work since it concerns the enrichment of non-equivariant computations with the natural torus action on the permutohedron. 
\end{remark}

To keep track of all classes at the same time, we will consider the following graded classes.
\begin{definition}\label{weighted}
    We define the \emph{graded total Chern sub-class}  $c([\Sbundle_M],x)\in H_T(\Pi_n)[x]$ of a matroid $M$ to be
    \[
    c([\Sbundle_M],x):=\sum_{j=0}^{\rank(M)} c_j([\Sbundle_M]) x^j.
    \]
    Similarly we have the  \emph{graded total Chern quotient-class} \ $c([\Qbundle_M],y)\in H_T(\Pi_n)[y]$ of a matroid which is 
    \[
    c([\Qbundle_M],y):=\sum_{j=0}^{\crk(M)} c_j([\Qbundle_M]) y^j.
    \]
\end{definition}

\subsection{The representable case}
\label{subsec:representable}

A matroid  $M$ on a groundset $[n]$ is called \emph{representable} over a field $\K$ if there exists a set of vectors $\{v_0,\dots,v_n\}$ inside some vector space over $\K$ such that $\{v_i\}_{i\in I}$ is independent if and only if $I$ is independent in the matroid $M$.

We now describe how one can associate a geometric structure to a matroid $M$. At first suppose that $M$ is representable by a set of vectors $\{\Bar{e}_0,\dots, \Bar{e}_n\}$ in a quotient space $\C^{n+1} / L$ for some $r$-dimensional space $L \subset \C^{n+1}$.
To the subspace $L$ of $\C^{n+1}$ we associate an equivariant bundle $\Sbundle_L$.
This bundle is a subbundle of an equivariant trivial bundle isomorphic to $\Pi_n\times \C^{n+1}$ with $t_i\in T_{n+1}$ acting on $\C^{n+1}$ by scaling $e_i$ by $t_i^{-1}$. As seen in Section \ref{variet}, we have a morphism $\Phi : T \to T/\C^* \hookrightarrow \Pi_n$ of the $(n+1)$-dimensional torus to the permutohedral variety. We start by defining $\Sbundle_L$ to be $L\subset \C^{n+1}$ over $\Phi(1: \ldots : 1),$ and use the torus action of $T$ on $L$ to extend it to the whole image of $\Phi$, by putting the vector space $t \cdot L := \{(t_0^{-1}v_0,\ldots,t_n^{-1}v_n) \mid v\in L\}$ over the point $\Phi(t)$. Note that if $t, t' \in T$ differ by a constant factor $c \in C^*$, then $t\cdot L = t' \cdot L$, so this is well-defined. 
Since $\Phi(T)$ is dense in $\Pi_n$, for arbitrary $p \in \Pi_n$ there exists a sequence $(t^{(m)})_m \subset T$ such that $t_m \to p$. We get a corresponding sequence $(t^{(m)}\cdot L)$ in the Grassmannian $\Gr(r, \C^{n+1})$. It is a remarkable property of the permutohedral variety that this limit does not depend on the choice of the sequence $(t^{(m)})_m$, see \cite[Lemma 3.5]{tautological-classes}. 
By this we get a vector bundle $\Sbundle_L$ over the permutohedral variety.  
By doing this for the quotient space $\C^{n+1} / L$ instead of $L$, we get a vector bundle $\Qbundle_L$, which is also the dual bundle of $\Sbundle_L$.

In the following lemma we describe the characters acting on the bundle $\Sbundle_L$ over the fixed points. In other words we have to determine what subspace of $\C^{n+1}$ corresponds to the fiber of $\Sbundle_L$ over a fixed point and read off the characters of the torus action on this space.

\begin{lemma}\label{subbundle}
In description of $H_T(\Pi_n)$ via localization as in Proposition \ref{H(Pi)},
the first equivariant Chern class $c_1(\Sbundle_L)$ is equal to $(-\sum_{i\in B_\sigma}t_i)_\sigma$. In other words, if $L_\sigma$ is the fiber of $\Sbundle_L$ over $q_\sigma$, then 
\[
L_\sigma=\Span(e_i \mid i\in B_\sigma ).
\]
\end{lemma}

\noindent We will only sketch a proof here. For a formal proof we refer to \cite[3.7]{tautological-classes}.

We look at a torus fixed point $q_\sigma$ corresponding to a permutation $\sigma \in S_{n+1}$.
Consider the ordered basis $e_{\sigma(0)},e_{\sigma(1)},\ldots, e_{\sigma(n)}$ of $\C^{n+1}$. 
Recall that we can reach $q_\sigma$ as a limit of a curve $c_\sigma$ as described in Remark \ref{limit}. Note that the curve is connecting the point $\Phi(1:\ldots:1)$ with $q_\sigma$ and we know the fiber over $\Phi(1: \ldots : 1)$. All we have to do is to act on the fiber along the segment joining the two points and see what we get in the limit. Namely on $e_{\sigma(i)}$ we act with $t^{-\sigma^{-1}(\sigma(i))}=t^{-i}$. 

Write the spanning vectors of $L$ in the rows of a matrix and take the reduced row echelon form of the matrix. By definition $B_\sigma=:\{i_1,\ldots,i_r\}$ is the set of indices of columns that contain a leading 1 in the reduced row echelon form. 
Consider the $k$-th row of our matrix, where $k \leq r=\dim(L)$: 
\[
(0,\ldots,0,1,*,\ldots,*)
\]
where $1$ appears on position $i_k$. 
Acting on the span of this vector and taking $t\to \infty$ we obtain $\Span(e_{i_k})$, so $e_{i_k}\in L_\sigma$. Performing the above for all $k$ we see that the subspace $L$ degenerates to the subspace spanned by $(e_{i_1},e_{i_2},...,e_{i_r})$.

\section{Pushforward of graded total chern classes}\label{sec:push}

We consider the result in this section as an equivariant analog of \cite[Theorem A]{tautological-classes}.
We will establish a closed formula for the pushforward of product of $c([\Sbundle_M^\vee],z)c([\Qbundle_M],w)$ along the map $\pi^n: X_E \to \P^n\times\P^n$ (see formula \eqref{MAP} in Section \ref{pxp} for the definition of $\pi^n$ and Definition \ref{weighted} for graded Chern classes). 
Adapting the approach of \cite[§4]{tautological-classes}, we will accomplish this by proving analogous deletion-contraction relations for both the pushforward and the desired closed formula, allowing for a straightforward inductive proof. Relating to notation from Proposition \ref{gysin} we will abbreviate $\tang_{\sigma}$ for $\tang_{q_\sigma}$ for a fixed point $q_\sigma\in \Pi_n$. Similarly we set $\tang_{a,b}$ to be  the product of the characters acting on the tangent space over a point in $\P(\C^{n+1})\times \P(\bigwedge\nolimits^n \C^{n+1})$ indexed by $(a,b)\in [n]\times [n]$ (see Remark \ref{indexing} for more about indexing).

\begin{theorem}\label{thrm:pushforward-graded-chern}
    Let $M$ be a matroid on a ground set $E$ of cardinality $n+1$, where $n\geq0$
    . Then for any $a,b\in E$, we have
    \[
        \pi^n_*\Big(c([\Sbundle_M^\vee],z)c([\Qbundle_M],w)\Big)_{a,b} = F_M(-t_a,t_b,z,w),
    \]
    where $F_M\in \Z[E][\alpha,\beta,z,w]$ is a polynomial given by
    \begin{align*}
        F_M(\alpha,\beta,z,w) &:= \frac{1}{\alpha+\beta} \sum_{S\subset{E}} (1-\alpha z)^{\rank(M)-\rank_M(S)}(1+\alpha w)^{\crk(M)-\nullity_M(S)}(1+\beta  z)^{\rank_M(S)}\cdot{}\\&\hskip6cm{}\cdot(1-\beta w)^{\nullity_M(S)}\prod_{e\in S} (\alpha+t_e)\prod_{e\notin S}(\beta-t_e).
    \end{align*}
\end{theorem}
Observe that $F_M$ is indeed a polynomial, since mod $\alpha+\beta$, the sum becomes
\begin{multline*}
    \sum_{S\subseteq E}(1+\beta z)^{\rk(M)}(1-\beta w)^{\crk(M)}\prod_{e\in E}(\beta-t_e)\cdot(-1)^{|S|} =\\= \left[(1+\beta z)^{\rk(M)}(1-\beta w)^{\crk(M)}\prod_{e\in E}(\beta-t_e)\right]\cdot(1-1)^{|E|} = 0.
\end{multline*}
Note that when viewing the cohomology of $\P^n\times\P^n$ as \eqref{eq:middle-guy-cohomology}, the Theorem simply states that
\[
\pi^n_*\Big(c([\Sbundle_M^\vee],z)c([\Qbundle_M],w)\Big) = F_M(\alpha,\beta,z,w)
\]
holds in $H_T(\P^n\times\P^n)[z,w]$.
For convenience, let us denote $\xi_M := c([\Sbundle_M^\vee],z)c([\Qbundle_M],w)$. We first establish a deletion-contraction relation for $F_M$:

\begin{proposition}
    \label{prop:FM-deletion-contraction}
    Let $\hat e\in E$. Then
    \[
        F_M = \begin{cases}
            (1+\alpha w)(\beta-t_{\hat e})F_{M\setminus\hat e} + (1+\beta z)(\alpha+t_{\hat e})F_{M\contract\hat e}, & \text{if $\hat e$ is a general element in $M$,}\\
            (1-t_{\hat e}w)(\alpha+\beta)F_{M\delete\hat e},& \text{if $\hat e$ is a loop in $M$,}\\
            (1+t_{\hat e}z)(\alpha+\beta)F_{M\contract\hat e},& \text{if $\hat e$ is a coloop in $M$.}
        \end{cases}
    \]
\end{proposition}
\begin{proof}
    In the sum over $S\subseteq E$ defining $F_M$, let us split the terms into a sum over ${\hat e}\notin S\subseteq E$ and another sum over ${\hat e}\in S\subseteq E$.
    
    Let us first work in the case when ${\hat e}$ is general. For ${\hat e}\notin S\subseteq E$,  we then have $\rank_M(S)=\rank_{M\delete {\hat e}}(S)$, and thus $\nullity_M(S)=\nullity_{M\delete {\hat e}}(S)$, as well as $\rank(M)=\rank(M\delete {\hat e})$ and $\crk(M)=\crk(M\delete\hat{e})+1$, giving
    \begin{multline*}
        \frac{1}{\alpha+\beta} \sum_{\hat e\notin S\subset{E}} (1-\alpha z)^{\rank(M)-\rank_M(S)}(1+\alpha w)^{\crk(M)-\nullity_M(S)}(1+\beta  z)^{\rank_M(S)}(1-\beta w)^{\nullity_M(S)}\cdot{}\\{}\cdot\prod_{e\in S} (\alpha+t_e)\prod_{e\in E\setminus S}(\beta-t_e).
    \end{multline*}
    \begin{multline*}
        = (1+\alpha w)(\beta-t_e)\frac1{\alpha+\beta}\sum_{S\subseteq E\setminus\{{\hat e}\}} (1-\alpha z)^{\rk(M\delete \hat e)-\rk_{M\delete {\hat e}}(S)}(1+\alpha w)^{\crk(M\delete {\hat e})-\nullity_{M\delete {\hat e}}(S)}\cdot{}\\{}\cdot(1+\beta z)^{\rk_{M\delete {\hat e}}(S)}(1-\beta w)^{\nullity_{M\delete {\hat e}}(S)}\prod_{e\in S}(\alpha+t_e)\prod_{e\in (E\setminus\{{\hat e}\})\setminus S}(\beta-t_e)
    \end{multline*}
    \begin{multline*}
        = (1+\alpha w)(\beta-t_{\hat e}) F_{M\delete {\hat e}}.\\
    \end{multline*}
    In a similar vein, if we consider ${\hat e}\in S\subseteq E$ and denote $S':=S\setminus\{{\hat e}\}$, we observe $\rk_M(S)=\rk_{M\contract {\hat e}}(S')+1$, $\nl_M(S)=\nl_{M\contract {\hat e}}(S')$ (in particular, this holds for $S=E$) and obtain
    \goodbreak
    \begin{multline*}
        \frac{1}{\alpha+\beta} \sum_{\hat e\in S\subset{E}} (1-\alpha z)^{\rank(M)-\rank_M(S)}(1+\alpha w)^{\crk(M)-\nullity_M(S)}(1+\beta  z)^{\rank_M(S)}(1-\beta w)^{\nullity_M(S)}\cdot{}\\{}\cdot\prod_{e\in S} (\alpha+t_e)\prod_{e\notin S}(\beta-t_e).
    \end{multline*}
    \begin{multline*}
        =(1+\beta z)(\alpha+t_{\hat e}) F_{M\contract {\hat e}},\\
    \end{multline*}
    which finishes the proof of the case when ${\hat e}$ is a general element.

    Since the loop and coloop cases are very analogous, let us only show the former. As in the previous case, the ${\hat e}\notin S\subseteq E$ terms contribute $(1+\alpha w)(\beta-t_{\hat e})F_{M\delete {\hat e}}$. But since we know ${\hat e}$ is a loop, we may relate the $\hat e\in S\subseteq E$ terms to $F_{M\delete {\hat e}}$ as well. In particular, denoting $S'=S\setminus\{{\hat e}\}$, we get $\rank_M(S)=\rank_{M\delete {\hat e}}(S')$ and $\nullity_M(S)=\nullity_{M\delete {\hat e}}(S')+1$ (again, this holds for $S=E$ as a special case),  so the ${\hat e}\in S\subseteq E$ terms contribute
    \[
        (1-\beta w)(\alpha+t_{\hat e})F_{M\delete {\hat e}},
    \]
    giving
    \[
        F_M = \Big((1+\alpha w)(\beta-t_{\hat e})+(1-\beta w)(\alpha+t_{\hat e})\Big)F_{M\delete {\hat e}} = (1-t_{\hat e}w)(\alpha+\beta)F_{M\delete {\hat e}}
    \]
    in total.
\end{proof}

Now we focus on establishing an analogous relation for $\pi^n_*(\xi_M)$. Since in the map $\pi^n:X_E\to \P^n\times \P^n$, the torus-fixed point of $X_E$ corresponding to a permutation $\tilde\sigma\in \Bij(E)$ will be mapped to the torus-fixed point of $\P^n\times\P^n$ corresponding to the pair $(\tilde\sigma(n), \tilde\sigma(0))$, we write the pushforward in terms of the localizations at fixed points using the respective tangent directions as
\begin{equation}
    \label{eq:pushforward-localization}
    \pi^n_*(\xi_M)_{a,b} = \tang^{\P^n\times\P^n}_{a,b}\sum_{\substack{\tilde\sigma\in\Bij(E)\\\tilde\sigma(n)=a\\\tilde\sigma(0)=b}} \frac{(\xi_M)_{\tilde\sigma}}{\tang^{X_E}_{\tilde\sigma}}.
\end{equation}
In the following, we will treat permutations as ordered tuples of elements. The main idea for the deletion-contraction relation for this pushforward is that we group together the terms for those $\tilde\sigma\in\Bij(E)$ that, as tuples, yield the same tuple $\sigma\in \Bij(E\setminus\{\hat e\})$ by omitting the element $\hat e$. For this, we use several observations from \cite{tautological-classes} connecting $\tilde\sigma$ and $\sigma$ in this situation. We denote by $\sigma^\ell$, for $\ell=0,\dots,n$, the permutation in $\Bij(E)$ that is obtained from $\sigma$ by inserting $\hat e$ in position $\ell$, i.e. so that $\sigma^\ell(\ell)=\hat e$.

\begin{lemma}[\cite{tautological-classes}, Lemma 4.4, Definition 4.5]
    \label{lem:bases-with-ksigma}
    Consider a matroid  $M$ on ground set $E$ and a permutation $\sigma\in\Bij(E\setminus\{\hat e\})$. Then there exists an index $k_\sigma=k_\sigma(M)\in\{-1,\dots,n\}$ such that
    \begin{itemize}
        \item $B_{\sigma^\ell}(M) = B_\sigma(M\contract \hat e)\sqcup\{\hat e\}$ whenever $\ell\leq k_\sigma$,
        \item $B_{\sigma^\ell}(M) = B_\sigma(M\delete \hat e)$ whenever $\ell>k_\sigma$, and
        \item $B_\sigma(M\contract \hat e)\sqcup\{k_\sigma\}=B_\sigma(M\delete \hat e)$.
    \end{itemize}
    Note that $\hat e$ is a loop exactly when $k_\sigma = -1$, and a coloop exactly when $k_\sigma = n$.
\end{lemma}

As an immediate consequence of these relations of bases, we can observe the behavior of Chern classes under this omission-insertion:
\begin{lemma}[\cite{tautological-classes},  Lemma 4.6]
    \label{lem:classes-with-ksigma}
    Let $M$ be a matroid on ground set $E\ni\hat e$ and consider $\sigma\in\Bij(E\setminus\{\hat e\})$.
    \begin{enumerate}
        \item For any $\ell=0,\dots,n$, we have
        \begin{align*}
            c([\Sbundle^\vee_M],z)_{\sigma^\ell} &=  \begin{cases}
                (1+t_{\hat e}z)\cdot c([\Sbundle^\vee_{M\contract\hat e}],z)_\sigma, &\text{if $\ell\leq k_\sigma(M)$,}\\
                c([\Sbundle^\vee_{M\delete\hat e}],z)_\sigma, &\text{if $\ell> k_\sigma(M)$,}
            \end{cases}\\
            c([\Qbundle_M],w)_{\sigma^\ell} &= \begin{cases}
                c([\Qbundle_{M\contract\hat e}],w)_\sigma, &\text{if $\ell\leq k_\sigma(M)$,}\\
                (1-t_{\hat e}w)\cdot c([\Qbundle_{M\delete\hat e}],w)_\sigma, &\text{if $\ell> k_\sigma(M)$,}
            \end{cases}\\
            (\xi_M)_{\sigma^\ell} &= \begin{cases}
                (1+t_{\hat e}z)\cdot(\xi_{M\contract\hat e})_\sigma, &\text{if $\ell\leq k_\sigma(M)$,}\\
                (1-t_{\hat e}w)\cdot(\xi_{M\delete\hat e})_\sigma, &\text{if $\ell> k_\sigma(M)$.}
            \end{cases}
        \end{align*}

        \item If $\hat e$  is a general element in $M$, i.e. if $k=k_\sigma(M)$ is neither $-1$ nor $n$, then
        \begin{align*}
            (1+t_{\sigma(k)}z)\cdot c([\Sbundle^\vee_{M\contract\hat e}],z)_\sigma &= c([\Sbundle^\vee_{M\delete\hat e}],z)_\sigma,\\
            (1-t_{\sigma(k)}w)\cdot c([\Qbundle_{M\delete\hat e}],w)_\sigma &= c([\Qbundle_{M\contract\hat e}],w)_\sigma.
        \end{align*}
    \end{enumerate}
\end{lemma}

With these lemmas, we may prove the deletion-contraction relation for the pushforward:
\begin{proposition}
    \label{prop:graded-chern-pushforward-deletion-contraction}
    Let $\hat e\in E$ be distinct from both of $a,b\in E$. Then
    \[
        \pi^n_*(\xi_M)_{a,b} = \begin{cases}
            (1-t_aw)(t_b-t_{\hat e})\pi^{n-1}_*(\xi_{M\delete \hat e})_{a,b} +{}& \text{if $\hat e$ is a general}\\ \qquad{}+(1+t_bz)(t_{\hat e}-t_a)\pi^{n-1}_*(\xi_{M\contract \hat e})_{a,b},
            &\text{\qquad element in $M$,}\\
            (1-t_{\hat e}w)(t_b-t_a)\pi^{n-1}_*(\xi_{M\delete \hat e})_{a,b},& \text{if $\hat e$ is a loop in $M$,}\\
            (1+t_{\hat e}z)(t_b-t_a)\pi^{n-1}_*(\xi_{M\contract \hat e})_{a,b},& \text{if $\hat e$ is a coloop in $M$.}\\
        \end{cases}
    \]
\end{proposition}
\begin{proof}
    We start by recalling
    \begin{equation}
        \tag{\ref{eq:pushforward-localization}}
        \pi^n_*(\xi_M)_{a,b} = \tang^{\P^n\times\P^n}_{a,b}\sum_{\substack{\tilde\sigma\in\Bij(E)\\\tilde\sigma(n)=a\\\tilde\sigma(0)=b}} \frac{(\xi_M)_{\tilde\sigma}}{\tang^{X_E}_{\tilde\sigma}}.
    \end{equation}
    We reorder the sum by setting $\tilde\sigma=\sigma^\ell$, where $\sigma$ runs through those permutations in $\Bij(E\setminus\{\hat e\})$ that satisfy $\sigma(n-1)=a$, $\sigma(0)=b$ and $\ell$ runs through $1,\dots,n-1$ (since $\hat e$ cannot be the first or the last element in $\tilde\sigma$).

    With this reordering, we wish to rewrite $(\xi_M)_{\sigma^\ell}$ and $\tang^{X_E}_{\sigma^\ell}$ more directly in terms of $\sigma$ and $\ell$. For $(\xi_M)_{\sigma^\ell}$, this is contained
    in Lemma~\ref{lem:classes-with-ksigma}(a), whereas for the tangent directions, we easily see
    \begin{align*}
        \tang^{X_E}_{\sigma^\ell} &= \prod_{i=1}^n (t_{\sigma^\ell(i-1)}-t_{\sigma^\ell(i)}) =\frac{(t_{\sigma(\ell-1)}-t_{\hat e})(t_{\hat e}-t_{\sigma(\ell)})}{t_{\sigma(\ell-1)}-t_{\sigma(\ell)}} \prod_{i=1}^{n-1} (t_{\sigma(i-1)}-t_{\sigma(i)}) =\\&=\frac{(t_{\sigma(\ell-1)}-t_{\hat e})(t_{\hat e}-t_{\sigma(\ell)})}{t_{\sigma(\ell-1)}-t_{\sigma(\ell)}}  \tang^{X_{E\setminus\{\hat e\}}}_\sigma.
    \end{align*}
    Similarly, we can also express 
    the tangent directions on $\P^n\times\P^n$ in relation to those on $\P^{n-1}\times\P^{n-1}$, where we use coordinates $E \setminus \{\hat{e}\}$:
    \[
        \tang^{\P^n\times\P^n}_{a,b} = \prod_{e\neq a}(t_e-t_a)\prod_{e\neq b}(t_b-t_e) =  (t_{\hat e}-t_a)(t_b-t_{\hat e})\cdot\tang^{\P^{n-1}\times\P^{n-1}}_{a,b}.
    \]

    At this point, it is advantageous to distinguish the loop and coloop cases:
    \begin{itemize}
        \item Case $k_\sigma(M)=-1$, i.e. $\hat e$ is a loop in $M$. Then all the terms in \eqref{eq:pushforward-localization} relate to classes corresponding to $M\delete \hat e$:
        \begin{align*}
            \pi^n_*(\xi_M)_{a,b} &= (t_{\hat e}-t_a)(t_b-t_{\hat e})\cdot\tang^{\P^{n-1}\times\P^{n-1}}_{a,b}\sum_{\substack{\sigma\in\Bij(E\setminus\{\hat e\})\\\sigma(n-1)=a\\\sigma(0)=b}}\sum_{\ell=1}^{n-1} \frac{(1-t_{\hat e}w)(\xi_{M\delete\hat e})_\sigma\cdot (t_{\sigma(\ell-1)}-t_{\sigma(\ell)})}{\tang^{X_{E\setminus\{\hat e\}}}_\sigma(t_{\sigma(\ell-1)}-t_{\hat e})(t_{\hat e}-t_{\sigma(\ell)})}=\\
            &= (1-t_{\hat e}w)\cdot \tang^{\P^{n-1}\times\P^{n-1}}_{a,b}\sum_{\substack{\sigma\in\Bij(E\setminus\{\hat e\})\\\sigma(n-1)=a\\\sigma(0)=b}} \frac{(\xi_{M\delete\hat e})_\sigma}{\tang^{X_{E\setminus\{\hat e\}}}_\sigma}  (t_{\hat e}-t_a)(t_b-t_{\hat e})\sum_{\ell=1}^{n-1} \left(\frac1{t_{\sigma(\ell-1)}-t_{\hat e}}-\frac1{t_{\sigma(\ell)}-t_{\hat e}}\right).
        \end{align*}
        In the last expression, we recognize a telescoping sum, allowing a simplification
        \begin{align*}
            \sum_{\ell=1}^{n-1} \left(\frac1{t_{\sigma(\ell-1)}-t_{\hat e}}-\frac1{t_{\sigma(\ell)}-t_{\hat e}}\right) = \frac1{t_{\sigma(0)}-t_{\hat e}}-\frac1{t_{\sigma(n-1)}-t_{\hat e}} =\frac1{t_{b}-t_{\hat e}}-\frac1{t_{a}-t_{\hat e}} =  \frac{t_b-t_a}{(t_b-t_{\hat e})(t_{\hat e}-t_a)}.
        \end{align*}
        Thus we finish this case with
        \begin{align*}
            \pi^n_*(\xi_M)_{a,b} &= (1-t_{\hat e}w)(t_b-t_a)\cdot \tang^{\P^{n-1}\times\P^{n-1}}_{a,b}\sum_{\substack{\sigma\in\Bij(E\setminus\{\hat e\})\\\sigma(n-1)=a\\\sigma(0)=b}} \frac{(\xi_{M\delete\hat e})_\sigma}{\tang^{X_{E\setminus\{\hat e\}}}_\sigma} = (1-t_{\hat e}w)(t_b-t_a)\cdot\pi^{n-1}_*(\xi_{M\delete\hat e}).
        \end{align*}

        \item Case $k_\sigma(M)=n$, i.e. $\hat e$ is a coloop in $M$. We proceed in almost the same way as in the previous case obtaining
        \begin{align*}
            \pi^n_*(\xi_M)_{a,b} &=
            (1+t_{\hat e}z)(t_b-t_a)\cdot\pi^{n-1}_*(\xi_{M\contract\hat e}).
        \end{align*}

    \item Case $1\leq k_\sigma(M)\leq n-1$, i.e. $\hat e$ is general in $M$. In accordance with Lemma~\ref{lem:classes-with-ksigma}(a), we split our sums into $\ell\leq k_\sigma$ and $\ell> k_\sigma$:
    \begin{align*}
        \pi^n_*(\xi_M)_{a,b} &= \tang^{\P^n\times\P^n}_{a,b}  \sum_{\substack{\sigma\in\Bij(E\setminus\{\hat e\})\\\sigma(n-1)=a\\\sigma(0)=b}}\left[
            \sum_{i=1}^{k_\sigma}\frac{(\xi_{M\contract\hat e})_\sigma}{\tang^{X_{E\setminus\{\hat e\}}}_\sigma}(1+t_{\hat e}z)\left(\frac1{t_{\sigma(\ell-1)}-t_{\hat e}}-\frac1{t_{\sigma(\ell)}-t_{\hat e}}\right)+{}\right.\\
            &\left.{}\hskip8em+\sum_{i=k_\sigma+1}^{n-1}\frac{(\xi_{M\delete\hat e})_\sigma}{\tang^{X_{E\setminus\{\hat e\}}}_\sigma}(1-t_{\hat e}w)\left(\frac1{t_{\sigma(\ell-1)}-t_{\hat e}}-\frac1{t_{\sigma(\ell)}-t_{\hat e}}\right)
        \right].
    \end{align*}
    We telescope the inner sums and obtain
    \begin{align}
        \pi^n_*(\xi_M)_{a,b} &= \tang^{\P^n\times\P^n}_{a,b}  \sum_{\substack{\sigma\in\Bij(E\setminus\{\hat e\})\\\sigma(n-1)=a\\\sigma(0)=b}}\left[
            \frac{(\xi_{M\contract\hat e})_\sigma}{\tang^{X_{E\setminus\{\hat e\}}}_\sigma}(1+t_{\hat e}z)\frac{t_{\sigma(k_\sigma)}-t_b}{(t_b-t_{\hat e})(t_{\sigma(k_\sigma)}-t_{\hat e})}+{}\right.\nonumber\\
            &\left.{}\hskip8em+\frac{(\xi_{M\delete\hat e})_\sigma}{\tang^{X_{E\setminus\{\hat e\}}}_\sigma}(1-t_{\hat e}w)\frac{t_a-t_{\sigma(k_\sigma)}}{(t_{\sigma(k_\sigma)}-t_{\hat e})(t_a-t_{\hat e})}
        \right]=\nonumber\\
        &=  \tang^{\P^{n-1}\times\P^{n-1}}_{a,b}  \sum_{\substack{\sigma\in\Bij(E\setminus\{\hat e\})\\\sigma(n-1)=a\\\sigma(0)=b}}\left[
            \frac{(\xi_{M\contract\hat e})_\sigma}{\tang^{X_{E\setminus\{\hat e\}}}_\sigma}(1+t_{\hat e}z)(t_{\hat e}-t_a)\frac{t_{\sigma(k_\sigma)}-t_b}{t_{\sigma(k_\sigma)}-t_{\hat e}}+{}\right.\nonumber\\
            &\left.{}\hskip8em+\frac{(\xi_{M\delete\hat e})_\sigma}{\tang^{X_{E\setminus\{\hat e\}}}_\sigma}(1-t_{\hat e}w)(t_b-t_{\hat e})\frac{t_{\sigma(k_\sigma)}-t_a}{t_{\sigma(k_\sigma)}-t_{\hat e}}
        \right].\label{eq:pushforward-before-subtraction}
    \end{align}
    At this point, let us subtract the right hand side
    \begin{align*}
         &\hphantom{{}={}}(1-t_aw)(t_b-t_{\hat e})\pi^{n-1}_*(\xi_{M\delete \hat e})_{a,b} +(1+t_bz)(t_{\hat e}-t_a)\pi^{n-1}_*(\xi_{M\contract \hat e})_{a,b}=\\
         &=\tang^{\P^{n-1}\times\P^{n-1}}_{a,b}  \sum_{\substack{\sigma\in\Bij(E\setminus\{\hat e\})\\\sigma(n-1)=a\\\sigma(0)=b}}\left[
            \frac{(\xi_{M\contract\hat e})_\sigma}{\tang^{X_{E\setminus\{\hat e\}}}_\sigma}(1+t_bz)(t_{\hat e}-t_a)+\frac{(\xi_{M\delete\hat e})_\sigma}{\tang^{X_{E\setminus\{\hat e\}}}_\sigma}(1-t_aw)(t_b-t_{\hat e})
        \right]
    \end{align*}
    of the deletion-contraction formula we wish to prove from \eqref{eq:pushforward-before-subtraction} and show the result will be zero. It turns out that this resulting ``error sum'' is actually term-wise zero, after we further simplify expressions using Lemma~\ref{lem:classes-with-ksigma}(b):
    \begin{align*}
        &\hphantom{{}={}}\pi^n_*(\xi_M)_{a,b}-\Big((1-t_aw)(t_b-t_{\hat e})\pi^{n-1}_*(\xi_{M\delete \hat e})_{a,b} +(1+t_bz)(t_{\hat e}-t_a)\pi^{n-1}_*(\xi_{M\contract \hat e})_{a,b}\Big)=\\
        &=  \tang^{\P^{n-1}\times\P^{n-1}}_{a,b}  \sum_{\substack{\sigma\in\Bij(E\setminus\{\hat e\})\\\sigma(n-1)=a\\\sigma(0)=b}}\left[
            \frac{(\xi_{M\contract\hat e})_\sigma}{\tang^{X_{E\setminus\{\hat e\}}}_\sigma}\left((1+t_{\hat e}z)(t_{\hat e}-t_a)\frac{t_{\sigma(k_\sigma)}-t_b}{t_{\sigma(k_\sigma)}-t_{\hat e}}  -  (1+t_bz)(t_{\hat e}-t_a) \right)+{}\right.\nonumber\\
            &\left.{}\hskip8em+\frac{(\xi_{M\delete\hat e})_\sigma}{\tang^{X_{E\setminus\{\hat e\}}}_\sigma}\left((1-t_{\hat e}w)(t_b-t_{\hat e})\frac{t_{\sigma(k_\sigma)}-t_a}{t_{\sigma(k_\sigma)}-t_{\hat e}}   -   (1-t_aw)(t_b-t_{\hat e}) \right)
        \right]=\\
        &= \tang^{\P^{n-1}\times\P^{n-1}}_{a,b}  \sum_{\substack{\sigma\in\Bij(E\setminus\{\hat e\})\\\sigma(n-1)=a\\\sigma(0)=b}} \frac{c([\Sbundle^\vee_{M\contract\hat e}],z)_\sigma c([\Qbundle_{M\delete\hat e}],w)_\sigma}{\tang^{X_{E\setminus\{\hat e\}}}_\sigma (t_{\sigma(k_\sigma)}-t_{\hat e})}\cdot{}\\
        &\hskip6em{}\cdot\left[
            (1-t_{\sigma(k_\sigma)}w)(t_{\hat e}-t_a)\Big((1+t_{\hat e}z)(t_{\sigma(k_\sigma)}-t_b) -  (1+t_bz)(t_{\sigma(k_\sigma)}-t_{\hat e})  \Big)+{}\right.\\
            &\left.\hskip6em{}+(1+t_{\sigma(k_\sigma)}z)(t_b-t_{\hat e})\Big( (1-t_{\hat e}w)(t_{\sigma(k_\sigma)}-t_a)    -  (1-t_aw)(t_{\sigma(k_\sigma)}-t_{\hat e})\Big)
        \right].
    \end{align*}
    Finally, it suffices to show that the last expression in brackets is identically  zero, which is easy to verify directly thanks to
    \begin{align*}
        (1+t_{\hat e}z)(t_{\sigma(k_\sigma)}-t_b) -  (1+t_bz)(t_{\sigma(k_\sigma)}-t_{\hat e}) &= -(1+t_{\sigma(k_\sigma)}z)(t_b-t_{\hat e}),\\
        (1-t_{\hat e}w)(t_{\sigma(k_\sigma)}-t_a)    -  (1-t_aw)(t_{\sigma(k_\sigma)}-t_{\hat e}) &= (1-t_{\sigma(k_\sigma)}w)(t_{\hat e}-t_a).\qedhere
    \end{align*}
    \end{itemize}
\end{proof}

\begin{proof}[Proof of Theorem~\ref{thrm:pushforward-graded-chern}]
    From Propositions~\ref{prop:FM-deletion-contraction} and \ref{prop:graded-chern-pushforward-deletion-contraction}, we see that $F_M(-t_a, t_b, z,w)$ and $\pi^n_*(\xi_M)_{a,b}$ satisfy the same deletion-contraction relations. Hence, once we verify their equality on
    base cases, an induction on the cardinality of $E$ will prove the theorem.

    Given $a$, $b$, as long as $E\setminus\{a,b\}$ is non-empty, we choose an $\hat e$ from it and apply the deletion-contraction relations along with the inductive hypothesis to prove the theorem. Thus, as base cases, it suffices to check the cases when $E=\{a,b\}$, which specifically only allows for $n\in\{0,1\}$, i.e. matroids on $1$ or $2$ elements:
    \begin{itemize}
        \item $n=0$. Then $X_E=\point=\P^n\times\P^n$ and the map $\pi^n$ is the identity. Suppose $E=\{e\}$ (so $a=b=e$), then
        \[
        \pi^0_*(\xi_M)_{e,e} = (\xi_M)_{(e)} = \begin{cases}
            1-t_ew, & \text{if $e$ is a loop,}\\
            1+t_e z, & \text{if $e$ is a coloop.}
        \end{cases}
        \]
        On the other hand, we see
        \begin{align*}
            F_m(\alpha,\beta,z,w) &= \begin{cases}
                \frac1{\alpha+\beta}\left((1+\alpha w)(\beta-t_e) + (1-\beta w)(\alpha+t_e)\right), & \text{if $e$ is a loop,}\\
                \frac1{\alpha+\beta}\left((1-\alpha z)(\beta-t_e) + (1+\beta z)(\alpha+t_e)\right), & \text{if $e$ is a coloop}
            \end{cases}\\
            &= \begin{cases}
                1-t_ew, & \text{if $e$ is a loop,}\\
                1+t_ez, & \text{if $e$ is a coloop,}
            \end{cases}
        \end{align*}
        so indeed $F_M(-t_e, t_e, z,w) = \pi^0_*(\xi_M)_{e,e}$.

        \item $n=1$.
        If it were the case that $a=b$, we could choose an $\hat e\in E\setminus\{a\}$ and apply the deletion-contraction relation, so it suffices to check the case when $a\neq b$. In \eqref{eq:pushforward-localization}, the conditions $\tilde\sigma(1)=a$, $\tilde\sigma(0)=b$ completely determine the permutation as $\tilde\sigma=(b,a)$, so 
        \[
            \pi^1_*(\xi_M)_{a,b} = (t_b-t_a)^2\cdot\frac{(\xi_M)_{(b,a)}}{t_b-t_a} = (t_b-t_a)(\xi_M)_{(b,a)}.
        \]
        On the other hand, in $F_M(-t_a, t_b, z,w)$, we are dividing by the non-zero element $t_b-t_a$, so we may ignore zero terms in the sum, i.e. all terms except $S=\{b\}$. Thus we get
        \begin{align*}
            F_M(-t_a, t_b,z,w) &= \frac1{t_b-t_a}\cdot (1+t_az)^{\rk_M(\{b,a\})-\rk_M(\{b\})}
            (1-t_aw)^{\nl_M(\{b,a\})-\nl_M(\{b\})}
            \cdot{}\\&\hskip4em{}\cdot
            (1+t_bz)^{\rk_M(\{b\})}
            (1-t_bw)^{\nl_M(\{b\})}
            \cdot(t_b-t_a)^2.
        \end{align*}
        Since
        \begin{align*}
            (1+t_az)^{\rk_M(\{b,a\})-\rk_M(\{b\})}
            (1-t_aw)^{\nl_M(\{b,a\})-\nl_M(\{b\})} &= \begin{cases}
                1+t_az, &\text{if $a\in B_{(b,a)}(M)$,}\\
                1-t_aw, &\text{if $a\notin B_{(b,a)}(M)$}
            \end{cases}
        \end{align*}
        and similarly
        \begin{align*}
            (1+t_bz)^{\rk_M(\{b\})}
            (1-t_bw)^{\nl_M(\{b\})} &= \begin{cases}
                1+t_bz, &\text{if $b\in B_{(b,a)}(M)$,}\\
                1-t_bw, &\text{if $b\notin B_{(b,a)}(M)$,}
            \end{cases}
        \end{align*}
        we conclude that indeed
        $F_M(-t_a,t_b,z,w) = (t_b-t_a)(\xi_M)_{(b,a)} = \pi^1_*(\xi_M)_{a,b}$.\qedhere
    \end{itemize}
\end{proof}

\section{Equivariant Tutte polynomial}
\label{chapter 4}
In this section, we define a polynomial invariant of matroids that generalizes the usual Tutte polynomial to an equivariant context, and relate it to the pushforward of $c([\Sbundle_M^\vee],z)c([\Qbundle_M],w)$. We are motivated in this by a non-equivariant result of \cite{tautological-classes}, Theorem A. There the authors prove that
\begin{multline}
\label{tautological pushforward}
    \int_{\Pi_n} c([\Sbundle_{U_{n,E}}^\vee],x)c([\Qbundle_{U_{1,E}}],y)c([\Sbundle_M^\vee],z)c([\Qbundle_M],w)=\\=\frac{(y+z)^{\rank(M)}(x+w)^{\crk(M)}}{x+y}T_M\left(\frac{x+y}{y+z},\frac{x+y}{x+w}\right),
\end{multline}
where $\int_{X_E}$ denotes the degree map (i.e. pushforward to a point) and everything is considered non-equivariantly.

The left-hand side can be viewed as a non-equivariant analog of $F_M(\alpha,\beta,z,w)$. As we proved in Section~\ref{sec:push}, the polynomial $F_M$ is the pushforward of $c([\Sbundle_M^\vee],z)c([\Qbundle_M],w)$ to $\P^n\times \P^n$, with the variables $\alpha$, $\beta$ corresponding to pullbacks of hyperplanes coming from the two projective spaces. In contrast, \eqref{tautological pushforward} is obtained by first intersecting with pullbacks of $\alpha$, $\beta$ to $X_E$ a number of times and then pushing forward to a point. Since non-equivariantly intersecting with $\alpha$, $\beta$ to a power higher than the dimension of a subvariety will give zero, one may think of the repeated intersection with $\alpha$, $\beta$ as a way to eliminate terms of high degree, while the degree map annihilates terms of low degree. Hence, varying the exponents $i$, $j$ in $\alpha^i,\beta^j$ will precisely give the coefficients of the pushforward of $c([\Qbundle_M],z)c([\Sbundle_M],w)$ in the non-equivariant cohomology ring of $\P^n\times \P^n$ extended with two formal variables $z,w$, namely in $\left(\Z[r,s]/( r^{n+1},s^{n+1})\right)[z,w]$.
These are therefore precisely the coefficients of the non-equivariant version of $F_M(r,s,z,w)$ read as a polynomial in $r$, $s$, i.e. the polynomial in $r$, $s$ that we obtain when we substitute all $t_e=0$ in the $F_M$. However, note that the grading is reversed: The term that comes from intersecting with $\alpha^i\beta^j$ will give the coefficient of $r^{n+1-i}s^{n+1-j}$. Therefore we will reverse the grading of the first two variables $r$, $s$ in $F_M(r,s,z,w)$ by replacing them by $\frac{1}{r},\frac{1}{s}$ and cancelling the denominators. In other words, we will now consider \begin{align*}
    (rs)^{|E|-1}F_M\left(\frac{1}{r},\frac{1}{s},z,w\right)\in \Z[E][r,s,z,w].
\end{align*}

From \eqref{tautological pushforward} we know that after substituting all $t_e=0$ this will exactly be \begin{align*}
   \left.(rs)^{|E|-1}F_M\left(\frac{1}{r},\frac{1}{s},z,w\right)\right|_{t_e = 0,e\in E}= \frac{(s+z)^{\rank(M)}(r+w)^{\corank(M)}}{r+s}T_M\left(\frac{r+s}{s+z},\frac{r+s}{r+w}\right).
\end{align*}
We will now use this to motivate a definition of an equivariant verion of the Tutte polynomial of a matroid $M$ by not substituting the $t_e$-variables to zero. The polynomial $\widehat{T}_M(x,y,r,s)$ that we want to define should therefore satisfy \begin{align}
\label{want equivariant tutte}
\widehat{T}_M\left(\frac{r+s}{s+z},\frac{r+s}{r+w},r,s\right)&=(rs)^{|E|-1}    \frac{r+s} {(s+z)^{\rank(M)}(r+w)^{\corank(M)}}F_M\left(\frac{1}{r},\frac{1}{s},z,w\right).
\end{align}
Unraveling the closed form of $F_M$, a straightforward manipulation reveals the following: 
\begin{align*}
    F_M\left(\frac{1}{r}, \frac{1}{s}, z , w\right) &= 
    \frac{1}{\frac{1}{r}+\frac{1}{s}} \sum_{S\subset E} \bigg[\left(1+\frac{z}{s}\right)^{\rank_M(S)} \left(1-\frac{w}{s}\right)^{\nullity_M(S)}
        \left(1-\frac{z}{r}\right)^{\rank(M)-\rank_M(S)}
        \\
        &\hskip3cm\left(1+\frac{w}{r}\right)^{\corank(M)-\nullity_M(S)}\prod_{e\notin S}\left(\frac{1}{s}-t_e\right)
        \prod_{e\in S} \left(t_e+\frac{1}{r}\right)\bigg]\\
    &=\frac{rs}{r + s}\cdot \frac{1}{(rs)^{|E|}}\sum_{S\subset E} \bigg[(s+z)^{\rank_M(S)}(s-w)^{\nullity_M(S)}
        \left(r-z\right)^{\rank(M)-\rank_M(S)}
        \\&\hskip4cm\left(r+w\right)^{\corank(M)-\nullity_M(S)}
        \prod_{e\notin S}\left(1-s t_e\right)
        \prod_{e\in S} \left(1+r t_e\right)\bigg]\\
     &= \frac{(s + z)^{\rank(M)}
         (r + w)^{\corank(M)}}{(rs)^{|E|-1}(r + s)} \sum_{S\subset E}\bigg[
         \left(\frac{r-z}{s + z}\right)^{\rank(M)-\rank_M(S)}
         \left(\frac{s-w}{r + w}\right)^{\nullity_M(S)}\\
        &\hskip6cm\prod_{e\notin S} \left(1-s t_e\right)
        \prod_{e\in S} \left(1+r t_e\right)\bigg]\\
     &= \frac{(s + z)^{\rank(M)}
         (r + w)^{\corank(M)}}{(rs)^{|E|-1}(r + s)} \sum_{S\subset E} \bigg[
         \left(\frac{r+s}{s + z}-1\right)^{\rank(M)-\rank_M(S)}
         \left(\frac{r+s}{r + w}-1\right)^{\nullity_M(S)}\\
         &\hskip6cm\prod_{e\notin S} \left(1-s t_e\right)
         \prod_{e\in S} \left(1+r t_e\right)\bigg].
\end{align*}
Hence, after substituting back into \eqref{want equivariant tutte}, all the factors cancel and we get
\begin{align*}
    \widehat{T}_M\left(\frac{r+s}{s+z},\frac{r+s}{r+w},r,s\right)&=\sum_{S\subset E}\bigg[
         \left(\frac{r+s}{s + z}-1\right)^{\rank(M)-\rank_M(S)}
         \left(\frac{r+s}{r + w}-1\right)^{\nullity_M(S)}\\&\hskip4cm
         \prod_{e\notin S} \left(1-s t_e\right)
        \prod_{e\in S} \left(1+r t_e\right)\bigg].
\end{align*}
Therefore, we propose the following definition of an equivariant version of the Tutte polynomial associated to a matroid $M$:
\begin{definition}[Equivariant Tutte polynomial]
    \label{equivariant-tutte-definition}
    Let $M$ be a matroid with ground set $E$. Then we define the \emph{equivariant Tutte polynomial} of $M$ as the following four-variable polynomial with coefficients in $\Z[E]$:
    $$
        \widehat{T}_M(x,y,r,s) = \sum_{S\subseteq E} (x-1)^{\rank(M)-\rank_M(S)}(y-1)^{\nullity_M(S)}\prod_{e\in S}(1+r t_e) \prod_{e\notin S} (1+s t_e).
    $$
    By convention, we set $ \widehat{T}_M(x,y,r,s)=1\in \Z$ for $M$ supported on the empty ground set.
\end{definition}
Note that setting either $r$ and $s$, or $t_e$ for all $e$ to zero recovers the usual Tutte polynomial.

For a positive real parameter $q$, the \emph{multivariate Tutte polynomial}
\begin{equation}
\label{potts}
\widetilde{Z}_M(q,\textbf{v})=q^{-\rank(M)}\widehat{T}_M(q+1,2,1,0)|_{t_e:=v_e-1}\in \Z[v_e\mid e\in E]
\end{equation}
 has already been studied in \cite{PottsSokal} and is related to the Potts model in statistical physics. In \cite[Theorem 4.10]{LorentzianPolynomials} the authors show that for $0<q\leq 1$ the homogenization of \eqref{potts} is a Lorentzian polynomial. \\
 Similarly we can obtain the equivariant Tutte polynomial from $\widetilde{Z}_M(q,\textbf{v})$, since both polynomials completely recover the  matroid $M$: \\
 \begin{align*}
     \widehat{T}_M(x,y,r,s)=(x-1)^{\rank(M)}\left(\prod_{e \in E}(1+st_e)\right)\widetilde{Z}_M((x-1)(y-1),((1+rt_e)(y-1)(1+st_e)^{-1})_{e \in E})
 \end{align*}
 In this way we give a geometric interpretation to the multivariate Tutte polynomial of \cite{PottsSokal}.

The usual Tutte polynomial satisfies a number of notable properties, for example passing to the dual matroid corresponds to exchanging the variables. Further, the Tutte polynomial is also multiplicative under direct sum of matroids. It turns out that the equivariant Tutte polynomial satisfies similar properties.
\begin{proposition}[dual matroid, sum of matroids]
    Let $M$ and $N$ be matroids and $M^*$ the dual of $M$. Then
    \begin{align*}
        \widehat{T}_{M^*}(x,y,r,s) &= \widehat{T}_M(y,x,s,r),\\
        \widehat{T}_{M\oplus N} &= \widehat{T}_M\cdot \widehat{T}_N.
    \end{align*}
\end{proposition}
\begin{proof}
We will prove the claim about the dual matroid first.
    Since $\rank_{M^*}(S) = |S| + \rank_M(E\setminus S) - \rank(M)$, we have
    \begin{align*}
        \rk(M^*)-\rk_{M^*}(S) &= |E\setminus S|-\rk_M(E\setminus S) = \nl_M(E\setminus S)\\
        \intertext{and}
        \nl_{M^*}(S) &= \rk(M)-\rk_M(E\setminus S)
    \end{align*}
    and so
    \begin{align*}
        \widehat T_{M^*}(x,y,r,s) &= \sum_{S\subseteq E} (x-1)^{\rank(M^*)-\rank_{M^*}(S)}(y-1)^{\nl_{M^*}(S)}\prod_{e\in S}(1+r t_e) \prod_{e\notin S} (1+s t_e) =\\
        &= \sum_{S\subseteq E} (y-1)^{\rk(M)-\rk_M(E\setminus S)} (x-1)^{\nl_M(E\setminus S)} \prod_{e\in E\setminus S}(1+st_e)\prod_{e\notin E\setminus S}(1+rt_e) =\\&= \widehat T_M(x,y,s,r).
    \end{align*}
    In the last equality we used the fact that summing over subsets is the same as summing over their complements.

Now we will prove the second claim for matroids $M$ and $N$. Set $P=M\oplus N$. If the ground sets of $M$ and $N$ are $E$ and $F$ respectively, then the ground set of $P$ may be identified with the disjoint union $E\cup F$.
If we consider $S\subseteq E\cup F$ and denote $S_1=S\cap E$, $S_2=S\cap F$, then $\rk_P(S) = \rk_M(S_1)+\rk_N(S_2)$. As a special case, this means $\rank(P)=\rank(M)+\rank(N)$. Lastly, since cardinality is additive with respect to disjoint unions, we also have
\[
\nl_P(S) = \nl_M(S_1)+\nl_N(S_2).
\]
Thus we obtain
\begin{align*}
     \widehat{T}_P(x,y,r,s) &=  \sum_{S\subseteq E\cup F} (x-1)^{\rank(P)-\rank_P(S)}(y-1)^{\nl_P(S)}\prod_{e\in S}(1+r t_e) \prod_{e\in (E\cup F)\setminus S} (1+s t_e) \\
   &=  \left(\sum_{S_1\subseteq E} (x-1)^{\rank(M)-\rank_M(S_1)}(y-1)^{\nl_M(S_1)}\prod_{e\in S_1}(1+r t_e) \prod_{e\in E\setminus S_1} (1+s t_e)\right) \cdot\\
   &\quad{}\cdot \left(\sum_{S_2\subseteq F} (x-1)^{\rank(N)-\rank_N(S_2)}(y-1)^{\nl_N(S_2)}\prod_{e\in S_2}(1+r t_e) \prod_{e\in F\setminus S_2} (1+s t_e)\right)\\
   &= \widehat T_M(x,y,r,s)\cdot \widehat T_N(x,y,r,s).
\end{align*}
The second equality holds because of the additive behaviors above and the bijective correspondence between subsets $S\subseteq E\cup F$ and pairs of subsets $S_1\subseteq E$, $S_2\subseteq F$.
\end{proof}

One of the definitions of the usual Tutte polynomial is by a deletion-contraction relation. Our equivariant Tutte polynomial satisfies a similar relation. In our case, the relation involves expressions that are linear in $t_e$, where $e$ is the element being deleted-contracted. Once again, by setting all $t_e$ to zero, one recovers the deletion-contraction relation of the usual Tutte polynomial.

\begin{proposition}[Deletion-contraction relation of equivariant Tutte polynomial]
\label{deletion-contraction for equivariant tutte}
The equivariant Tutte polynomial of a matroid $M$ satisfies the following relation with respect to deleting or contracting an edge $\hat{e}$ depending on whether $\hat{e}$ is a general element, a loop or a coloop:
\begin{enumerate}
\item If $\hat{e}$ is a general element, then
\[
\widehat{T}_M=(1+s t_{\hat{e}})\cdot\widehat{T}_{M \delete {\hat{e}}}+(1+rt_{\hat{e}})\cdot\widehat{T}_{M\contract{\hat{e}}}.
\]
    \item  
 If $\hat{e}\in E$ is a loop, then
\[
\widehat{T}_M=\Big((y-1)(1+r t_{\hat{e}})+(1+s t_{\hat{e}})\Big)\cdot \widehat{T}_{M \delete {\hat{e}}}.
\]
\item If ${\hat{e}}\in E$ is a coloop, then
\[
\widehat{T}_M=\Big((x-1)(1+s t_{\hat{e}})+(1+r t_{\hat{e}})\Big) \cdot\widehat{T}_{M  \contract  {\hat{e}}}.
\]
\end{enumerate}
\end{proposition}

\begin{proof}
First we resolve the case when ${\hat{e}}$ is a general element.  We split the sum indexed by $S\subseteq E$ in the definition of the equivariant Tutte polynomial according to whether ${\hat{e}}$ belongs to $S$ or not:
\begin{align}\label{delcon}
        \widehat{T}_M(x,y,r,s) &= \sum_{\substack{S\subseteq E \\ {\hat{e}}\notin S} } (x-1)^{\rank(M)-\rank_M(S)}(y-1)^{\nullity_M(S)}\prod_{e\in S}(1+r t_e) \prod_{e\notin S} (1+s t_e) +{} \\ \nonumber
        &\quad{}+\sum_{\substack{S\subseteq E \\ {\hat{e}}\in S} } (x-1)^{\rank(M)-\rank_M(S)}(y-1)^{\nullity_M(S)}\prod_{e\in S}(1+r t_e) \prod_{e\notin S} (1+s t_e).
\end{align}
If ${\hat{e}}\notin S$, then we can identify $S$ as a subset in the matroid $M\delete {\hat{e}}$. Note that $\rank(M \delete {\hat{e}} )=\rank(M)$ and $\rank_{M\delete\hat e}$ is merely a restriction of $\rank_M$, so if we compare the first sum in \eqref{delcon} to the definition of $\widehat T_{M\delete\hat e}$, we see that exponents are the same, meaning the only difference is the additional factor $1+st_e$. Thus the first sum adds up to $(1+s t_{\hat{e}})\widehat{T}_{M\setminus {\hat{e}}}$. 

Now we will similarly prove that the second sum in~\eqref{delcon} is equal to $(1+r t_{\hat{e}})\widehat{T}_{M/{\hat{e}}}$. To see this, note that the rank of $M$ decreases by one after contraction of ${\hat{e}}$. Under the natural identification of subsets of $S$ containing ${\hat{e}}$ of $M$ with subsets $S':=S\setminus\{\hat e\}$ of $M/{\hat{e}}$, we further have $|S|-1=|S'|$, $\rank_M(S)-1=\rank_{M/{\hat{e}}}(S')$. So as above, in comparing the second sum in~\eqref{delcon} to the definition of $\widehat T_{M\contract\hat e}$, the expressions differ only by an additional factor of $1+rt_e$, meaning we obtain $(1+rt_e)\widehat T_{M\contract\hat e}$. Adding the two sums yields our claim (a).

Next, let us assume that ${\hat{e}}$ is a loop. Then we have $\rank(M \delete {\hat{e}} )=\rank(M)$. Also, if ${\hat{e}}\notin S$, then $\rank_M(S)=\rank_{M\setminus {\hat{e}}}(S)$, so the first sum in \eqref{delcon}  is equal to $(1+s t_{\hat{e}})\widehat{T}_{M / {\hat{e}}}$ as before. As to the second sum, we can naturally identify a subset $S$ of $M$ containing ${\hat{e}}$ with the subset $S':=S\setminus \{\hat e\}$ of $M\setminus {\hat{e}}$. We then have $\rank_M(S)=\rank_{M\setminus \hat{e}}(S')$ but $|S|-1=|S'|$. Using this, we see that compared to the definition of $\widehat T_{M\delete\hat e}$, the second sum in \eqref{delcon} differs in the exponent of $y-1$ being greater by $1$ as well as in an additional factor of $1+rt_e$, so the second sum adds up to $(y-1)(1+rt_e)\widehat T_{M\delete\hat e}$.
Adding the two results together, we obtain our claim (b).

The case of a coloop proceeds analogously, so we omit it.
\end{proof}

Our next goal is to prove that the assignment $M\mapsto \widehat{T}_M(x,y,r,s)\in \Z[E][x,y,r,s]$ is valuative, i.e. is compatible with subdivisions of the base polytope of $M$ into base polytopes of smaller matroids. We quickly recall the following definition.
\begin{definition}
    \begin{enumerate}
        \item Let $M$ be a matroid of rank $r$ on the groundset $E$, then the \emph{base polytope of} $M$ is by definition \begin{align*}
            P(M):=\conv\left(\left\{\sum_{b\in B}e_b \mid B\text{ is a basis of } M \right\}\right)\subset \R^E
        \end{align*}
        where $e_b\in \R^E$ denotes the $b$-th standard basis vector. This polytope is contained in the affine subspace of vectors with coordinate sum $r$. We denote by $\mathbbm{1}_{P(M)}: \R^E\to \{0,1\}$ its indicator function which takes value one exactly when the input belongs to $P(M)$.
        \item Fix $r\in \N_0$, a finite set $E$ and an abelian group $G$. Denote by $\text{Mat}_r(E)$ the free abelian group generated by rank $r$ matroids on $E$. A group morphism $f:\text{Mat}_r(E)\to G$ is said to be \emph{valuative} if for every element $\sum_{M}z_MM\in \text{Mat}_r(M)$, where the sum ranges over all rank $r$ matroids on $E$ and $z_M\in\Z$ are integers, we have the following implication: \begin{align*}
            \sum_{M}z_M\mathbbm{1}_{P(M)}\equiv 0 \text{ on all of }\R^n \implies f\left(\sum_M z_MM\right)=0\in G
        \end{align*}
    \end{enumerate}
\end{definition}
For more details on valuativity we refer to \cite{Stel}. Here we will only prove the following result which follows from a valuativity result shown in \cite{tautological-classes}.
\begin{proposition}
    \label{valuativity}
    The map $M\mapsto \widehat{T}_M(x,y,r,s)\in \Z[E][x,y,r,s]$ assigning to a matroid its equivariant Tutte polynomial extends to a valuative group morphism $\text{Mat}_r(E)\to \Z[E][x,y,r,s]$ for every $r\in \N_0$ and every finite set $E$. The same is true  for the multivariate Tutte polynomial $\widetilde{Z}_M(q,\textbf{v})$ as well as any of the entries of table \ref{tab:eval_of_tutte} in Section \ref{sec:evaluations}, which are further explained in Theorem \ref{thm:evaluations}.
\end{proposition}
\begin{proof}
   Proposition 5.6 from \cite{tautological-classes} shows that the assignment $M\mapsto c([\Sbundle^\vee],a)c([\Qbundle_M],b)$ extends to a valuative morphism. Since the pushforward map is additive, we conclude that the same holds for the pushforward of this product to $\P^n\times \P^n$, that is for the assignment $M\mapsto F_M(a,b,z,w)\in H_T(\P^n\times \P^n)[a,b]$. Any sort of substitution preserves this property, but to obtain $\widehat{T}_M$ we also need to multiply by some factor. For this to preserve valuativity it is important that this factor does not depend on the matroid $M$ but only on its rank and on $E$, which are both fixed. The first step to get $\widehat{T}_M$ from $F_M$ was to revert the grading of the first two variables. To do so we substitute $a\to \frac{1}{r}, b \to \frac{1}{s}$ and then multiply by $(rs)^{|E|-1}$. Since this factor is independent of $M$ we see that the assignment $M\mapsto (rs)^{|E|-1}F_M(\frac{1}{r},\frac{1}{s},z,w)\in \Z[E][r,s,z,w]$ extends to a valuative group morphism. Next note that by \eqref{want equivariant tutte} \begin{align*}
      \widehat{T}_M\left(\frac{r+s}{s+z},\frac{r+s}{r+w},r,s\right)&=(rs)^{|E|-1}    \frac{r+s} {(s+z)^{\rank(M)}(r+w)^{\corank(M)}}F_M\left(\frac{1}{r},\frac{1}{s},z,w\right).
   \end{align*}
   Since the factor $\frac{r+s} {(s+z)^{\rank(M)}(r+w)^{\corank(M)}}$ depends only on the rank of $M$ and on $E$, this is still valuative. Lastly we perform the substitution $z\to \frac{r+s-sx}{x},w\to \frac{r+s-ry}{y}$ to obtain $\widehat{T}_M(x,y,r,s)$. As substitutions preserve valuativity, we get the desired result. \\
   The valuativity of $\widetilde{Z}_M(q,\textbf{v})$ and the functions in table \ref{tab:eval_of_tutte} now follows immediately since these are obtained via substitutions of $\widehat{T}_M(x,y,r,s)$, reading of coefficients and by multiplying by factors depending only on the rank of $M$.
\end{proof}

To close out this section, we examine some particular examples of the equivariant Tutte polynomial in families of matroids where we can express it in a simple form.
\begin{example}
    \label{all-loops-coloops}
    Let us a consider a matroid $M$ on ground set $E$ that contains no general elements, i.e. is a collection of loops and coloops. Denoting its set of loops as $L\subseteq E$, an easy induction via Proposition~\ref{deletion-contraction for equivariant tutte} (b), (c) then yields
    \[
        \widehat T_M(x,y,r,s)= \prod_{e\in L}\bigl(y+(yr-r+s)t_e\bigr)\prod_{e\notin L}\bigl(x+(xs-s+r)t_e\bigr).
    \]
\end{example}

\begin{example}
\label{Example:cycle}
    Next we express the equivariant Tutte polynomial of the corank $1$ uniform matroid $M=U_{n,E}$ on ground set $E$ of size $n+1$, that is the graphic matroid coming from a cycle of length $n+1$. In Definition~\ref{equivariant-tutte-definition}, each term corresponding to $S\subseteq E$ depends just on the set itself, its rank $\rk_M(S)$ and the overall rank $\rk(M)$. The rank function of $U_{n,E}$ agrees with that of $U_{n+1,E}$ for all $S$ except $S=E$, while the overall rank is smaller by $1$. Further, we know the equivariant Tutte polynomial of $U_{n+1,E}$ from the previous example, since it is just a collection of coloops. So, accounting for these discrepancies, we get 
    \begin{align*}
        \widehat T_{M} &= \frac{\widehat T_{U_{n+1,E}} - (x-1)^{0}(y-1)^0\prod_{e\in E}(1+rt_e)}{x-1} + (x-1)^0(y-1)^1\prod_{e\in E}(1+rt_e) =\\
        &= \frac1{x-1}\left(\prod_{e\in E} \bigl(x+(xs-s+r)t_e\bigr)-\prod_{e\in E}(1+rt_e)\right) + (y-1)\prod_{e\in E}(1+rt_e).
        \end{align*}
        Notice that uniform matroids are preserved under relabeling their groundset, hence this polynomial has to be symmetric in the $t_e$-variables. Indeed, if we factor out the product above to write it in terms of the monomial basis in the $t_e$-variables, we see that for each $A\subseteq E$ the coefficient of the monomial $t_A:=\prod_{e\in A}t_e$ only depends on $|A|$. This makes it easy to write the expression in terms of elementary symmetric polynomials.
        \begin{align*}
        \widehat T_{M}&=\sum_{A\subseteq E} \left(\frac{(xs-s+r)^{|A|}x^{|E\setminus A|}-r^{|A|}}{x-1}+(y-1)r^{|A|}\right)\prod_{e\in A}t_e =\\
        &=\sum_{k=0}^{n+1}\left(\frac{(xs-s+r)^{k}x^{n+1-k}-r^{k}}{x-1}+(y-1)r^{k}\right)\elem_k(\{t_e\}_{e\in E}).
    \end{align*}
    
\end{example}

\section{Equivariant Tutte-Grothendieck}
\label{section:equi-T_G}
Our tool to understand the pushforward of $c([Q_M],w)c([S_M^\vee],z)$ was by analyzing how these invariants change under deletion and contraction of an element in $M$. In general there is a way to express a matroid invariant that satisfies certain deletion-contraction-relations in a closed form. The classically known way to do this is via the following result from (\cite{tutte and applications}, Section 6.2).
\begin{proposition}[Generalized Tutte-Grothendieck invariants]
\label{classical T-G}
Associated to a matroid $M$ with finite ground set $E$, there is a unique polynomial $G_M(u,v,a,b,\gamma)\in \Z[u,v,a,b,\gamma]$ that satisfies the following recursive relations: \\ \underline{Base case:} If $E=\emptyset$ then \begin{align*}
        G_M(u,v,a,b,\gamma)=\gamma.
    \end{align*}
    \underline{Recursion:} If $E\neq \emptyset$ and $e\in E$ then \begin{align*}
        G_M(u,v,a,b,\gamma)=\begin{cases}
            aG_{M\delete e}(u,v,a,b,\gamma)+bG_{M\contract e}(u,v,a,b,\gamma), &\text{if $e$ is a general element in $M$,}\\
            vG_{M\delete e}(u,v,a,b,\gamma), &\text{if $e$ is a loop in $M$,} \\
            uG_{M\contract e}(u,v,a,b,\gamma) &\text{if $e$ is a coloop in $M$.}
        \end{cases}
    \end{align*}
    In closed form this polynomial is given as \begin{align*}
        G_M(u,v,a,b)=\gamma b^{\rank(M)}a^{\crk(M)}T_M\left(\frac{u}{b},\frac{v}{a}\right)\in \Z[u,v,a,b,\gamma],
    \end{align*}
    where \begin{align*}
        T_M(s,t)=\sum_{S\subseteq E}(s-1)^{\rank(M)-\rank_M(S)}(t-1)^{\nullity_M(S)}
    \end{align*}
    is the usual Tutte polynomial.
\end{proposition}
We say that some matroid invariant $G_M$ with values in some ring $R$ is a \emph{specialization of the Tutte polynomial} if we can find elements $x,y,z,w\in R$ such that $G_M=x^{\rank(M)}y^{\crk(M)}T_M(z,w)$. The statement above then says that every matroid invariant that satisfies deletion-contraction relations with $a,b\neq 0$ (in the notation of the Proposition) is a specialization of $T_M$. However, if for some matroid invariant we have $a=0$ or $b=0$, then this invariant cannot be obtained by multiplying some evaluation of $T_M(s,t)$ by a factor that only depends on the rank of the matroid. Such invariants can still be recovered from $T_M(s,t)$ but the process requires to read of certain coefficients rather than just performing substitutions.

Notice that for the invariants that we considered so far this result cannot be applied since the deletion-contraction relations for us linearly depend on the variable $t_e$ indexed by the element $e$ that we choose to remove from the matroid. For this situation we obtain the following extension of the usual result.

\begin{proposition}[Equivariant Tutte-Grothendieck]
\label{Equivariant T-G}\mbox{}
    \begin{enumerate}
        \item Associated to a matroid $M$ on the finite ground set $E$ there exists a unique 7-variable polynomial $H_M\in \Z[E][a_1,a_2,b_1,b_2,\alpha,\beta,\gamma]$ that satisfies the following recursion. \\\underline{Base case:} If $E=\emptyset$, then \begin{align*}
        H_M= \gamma.
        \end{align*}
        \underline{Recursion:} If $E\neq \emptyset$ and $e\in E$ then \begin{align*}
            H_M=\begin{cases}
                (a_1t_e+a_2)H_{M\delete e}+(b_1t_e+b_2)H_{M\contract e}, &\text{if $e$ is a general element in $M$,}\\
                ((b_1\alpha +a_1)t_e+(b_2\alpha+a_2))H_{M\delete e}, &\text{if $e$ is a loop in $M$,}\\
                ((a_1\beta+b_1)t_e+(a_2\beta+b_2))H_{M\contract e} &\text{if $e$ is a coloop in $M$.}
            \end{cases}
        \end{align*} 
        This polynomial is given in closed form as \begin{align*}
            H_M=\gamma a_2^{\crk(M)}b_2^{\rank(M)}\widehat{T}_M\left(\frac{\beta a_2}{b_2}+1, \frac{\alpha b_2}{a_2}+1,\frac{b_1}{b_2},\frac{a_1}{a_2}\right).
        \end{align*}
        \item The above result is optimal in the following sense: Let $R$ be any integral ring and fix $\gamma, a_1,a_2,b_1,b_2,c_1,c_2,d_1,d_2\in R$. Assume there exists a well-defined matroid invariant $G_M\in R[t_e\mid e\in E]$ such that the following recursion holds for all matroids $M$ over the ground set $E$.
        \\\underline{Base case:} If $E=\emptyset$, then \begin{align*}
            G_M= \gamma.
        \end{align*}
        \underline{Recursion:} If $E\neq \emptyset$ and $e\in E$, then \begin{align*}
            G_M=\begin{cases}
                (a_1t_e+a_2)G_{M\delete e}+(b_1t_e+b_2)G_{M\contract e}, &\text{if $e$ is a general element in $M$,}\\
                (c_1t_e+c_2)G_{M\delete e}, &\text{if $e$ is a loop in $M$,}\\
                (d_1t_e+d_2)G_{M \contract e}, &\text{if $e$ is a coloop in $M$.} 
            \end{cases}
        \end{align*} 
        Assume further that $a_2,b_2\neq 0$ and at least one of the parameters $c_1$, $c_2$, $d_1$, $d_2$ is not zero (otherwise $G_M=0$ for all non-empty $M$),
        then there exist $\alpha$, $\beta$ in the fraction field of $R$ such that we have \begin{align*}
            c_1&=b_1\alpha+a_1,\\
            c_2&=b_2\alpha+a_2,\\
            d_1&=a_1\beta+b_1,\\
            d_2&=a_2\beta+b_2
        \end{align*}
        and therefore \begin{align*}
            G_M=\gamma a_2^{\crk(M)}b_2^{\rank(M)}\widehat{T}_M\left(\frac{\beta a_2}{b_2}+1, \frac{\alpha b_2}{a_2}+1,\frac{b_1}{b_2},\frac{a_1}{a_2}\right).
        \end{align*}
        \end{enumerate}
\end{proposition}
We can see that part (a) naturally extends Proposition \ref{classical T-G} and hence further justifies calling $\widehat{T}_M$ the equivariant Tutte polynomial: Just like the usual Tutte polynomial is a universal deletion-contraction-invariant in the sense of Proposition \ref{classical T-G}, the equivariant Tutte polynomial is universal for deletion-contraction-relations that linearly depend on the deleted/contracted element in the sense of Proposition \ref{Equivariant T-G}.\\ Part (b) justifies the very specific coefficients in the relations, which come from the observation that when $M$ is a uniform matroid, then the resulting invariant has to be symmetric in the variables $t_{e}$. Note that just as in the non-equivariant version this statement says that almost every deletion-contraction invariant $H_M$ as in part (a) of Proposition \ref{Equivariant T-G} is a specialization of $\widehat{T}_M$. The only exceptions are the cases $a_2=0$ or $b_2=0$, which would require substitutions in the $t_e$-variables or even to read of coefficients instead of just doing evaluations of $\widehat{T}_M$. An example with $a_2=b_2=0$ is the multivariate Tutte polynomial $\widetilde{Z}_M(q,\textbf{v})$ which indeed required the substitution $t_e\to v_e-1$ as we have seen in \eqref{potts}.
\begin{proof}[Proof of Proposition \ref{Equivariant T-G}]
\begin{enumerate}
    \item By induction on the number $|E|$ it is clear that if such an invariant $H_M$ exists, it will be unique and polynomial in all occuring variables. To show existence it is therefore enough to check that the closed form that we claim for $H_M$ satisfies the correct deletion-contraction relations.  \\
    \underline{Base case:} Assume that $E=\emptyset$. In this case $M$ is the empty matroid with rank 0.  Then by definition $\widehat{T}_\emptyset=1$ and so we have \begin{align*}
    \gamma a_2^{\crk(M)}b_2^{\rank(M)}\widehat{T}_M\left(\frac{\beta a_2}{b_2}+1, \frac{\alpha b_2}{a_2}+1,\frac{b_1}{b_2},\frac{a_1}{a_2}\right)=\gamma  a_2^{0}b_2^{0} = \gamma
    \end{align*}
    \underline{Recursion:} Now assume $E\neq \emptyset$ and $e\in E$. We consider the usual three cases: \begin{enumerate}[label={(\arabic*)}]
        \item If $e$ is neither loop nor coloop in $M$, then deleting $e$ will preserve the rank of $M$ while contracting $e$ will reduce the rank by one. By applying the corresponding case of Proposition \ref{deletion-contraction for equivariant tutte} we obtain \begin{align*}
        &\gamma a_2^{\crk(M)}b_2^{\rank(M)}\widehat{T}_M\left(\frac{\beta a_2}{b_2}+1, \frac{\alpha b_2}{a_2}+1,\frac{b_1}{b_2},\frac{a_1}{a_2}\right)\\={}&
        \gamma a_2^{\crk(M)}b_2^{\rank(M)}\left(1+\frac{a_1}{a_2}t_e\right)\widehat{T}_{M\delete e}\left(\frac{\beta a_2}{b_2}+1, \frac{\alpha b_2}{a_2}+1,\frac{b_1}{b_2},\frac{a_1}{a_2}\right)+\\&+\gamma a_2^{\crk(M)}b_2^{\rank(M)}\left(1+\frac{b_1}{b_2}t_e\right)\widehat{T}_{M\contract e}\left(\frac{\beta a_2}{b_2}+1, \frac{\alpha b_2}{a_2}+1,\frac{b_1}{b_2},\frac{a_1}{a_2}\right)\\
        ={}&\gamma a_2^{\crk(M)-1}b_2^{\rank(M)}\left(a_1t_e+a_2)\right)\widehat{T}_{M\delete e}\left(\frac{\beta a_2}{b_2}+1, \frac{\alpha b_2}{a_2}+1,\frac{b_1}{b_2},\frac{a_1}{a_2}\right)+\\
        &+\gamma a_2^{\crk(M)}b_2^{\rank(M)-1}\left(b_1t_e+b_2\right)\widehat{T}_{M\contract e}\left(\frac{\beta a_2}{b_2}+1, \frac{\alpha b_2}{a_2}+1,\frac{b_1}{b_2},\frac{a_1}{a_2}\right)\\
        ={}&\left(a_1t_e+a_2\right)\left(\gamma a_2^{\crk(M\delete e)}b_2^{\rank(M\delete e)}\widehat{T}_{M\delete e}\left(\frac{\beta a_2}{b_2}+1, \frac{\alpha b_2}{a_2}+1,\frac{b_1}{b_2},\frac{a_1}{a_2}\right)\right)+\\
        &+\left(b_1t_e+b_2\right)\left(\gamma a_2^{\crk(M\contract e)}b_2^{\rank(M\contract e)}\widehat{T}_{M\contract e}\left(\frac{\beta a_2}{b_2}+1, \frac{\alpha b_2}{a_2}+1,\frac{b_1}{b_2},\frac{a_1}{a_2}\right)\right)\\
    \end{align*}
    as claimed.
        \item If $e$ is a loop in $M$, then deleting $e$ will not change the rank of $M$ and so we again just apply the corresponding case of Proposition \ref{deletion-contraction for equivariant tutte} to see \begin{align*}
    &\gamma a_2^{\corank(M)}b_2^{\rank(M)}\widehat{T}_M\left(\frac{\beta a_2}{b_2}+1, \frac{\alpha b_2}{a_2}+1,\frac{b_1}{b_2},\frac{a_1}{a_2}\right)\\={}&a_2\left(\frac{\alpha b_2}{a_2}\left(1+\frac{b_1}{b_2}t_e\right)+\left(1+\frac{a_1}{a_2}t_e\right)\right)\\&\gamma a_2^{\corank(M\delete e)}b_2^{\rank(M\delete e)}\widehat{T}_{M\delete e}\left(\frac{\beta a_2}{b_2}+1, \frac{\alpha b_2}{a_2}+1,\frac{b_1}{b_2},\frac{a_1}{a_2}\right)\\={}&(\alpha(b_1t_e+b_2)+(a_1t_e+a_2))\\&\gamma a_2^{\corank(M\delete e)}b_2^{\rank(M\delete e)}\widehat{T}_{M\delete e}\left(\frac{\beta a_2}{b_2}+1, \frac{\alpha b_2}{a_2}+1,\frac{b_1}{b_2},\frac{a_1}{a_2}\right)\\={}&((b_1\alpha+a_1)t_e+(b_2\alpha+a_2))\\&\gamma a_2^{\corank(M\delete e)}b_2^{\rank(M \delete e)}\widehat{T}_{M\delete e}\left(\frac{\beta a_2}{b_2}+1, \frac{\alpha b_2}{a_2}+1,\frac{b_1}{b_2},\frac{a_1}{a_2}\right)
        \end{align*}
        \item If $e$ is a coloop in $M$, then contracting $e$ will reduce the rank of $M$ by one and so similarly we get \begin{align*}
    &\gamma a_2^{\corank(M)}b_2^{\rank(M)}\widehat{T}_M\left(\frac{\beta a_2}{b_2}+1, \frac{\alpha b_2}{a_2}+1,\frac{b_1}{b_2},\frac{a_1}{a_2}\right)\\={}&b_2\left(\frac{\beta a_2}{b_2}\left(1+\frac{a_1}{a_2}t_e\right)+\left(1+\frac{b_1}{b_2}t_e\right)\right)\\&\gamma a_2^{\corank(M\contract e)}b_2^{\rank(M\contract e)}\widehat{T}_{M\contract e}\left(\frac{\beta a_2}{b_2}+1, \frac{\alpha b_2}{a_2}+1,\frac{b_1}{b_2},\frac{a_1}{a_2}\right)\\={}&(\beta(a_1t_e+a_2)+(b_1t_e+b_2))\\&\gamma a_2^{\corank(M\contract e)}b_2^{\rank(M\contract e)}\widehat{T}_{M\contract e}\left(\frac{\beta a_2}{b_2}+1, \frac{\alpha b_2}{a_2}+1,\frac{b_1}{b_2},\frac{a_1}{a_2}\right)\\={}&((a_1\beta+b_1)t_e+(a_2\beta+b_2))\\&\gamma a_2^{\corank(M\contract e)}b_2^{\rank(M\contract e)}\widehat{T}_{M\contract e}\left(\frac{\beta a_2}{b_2}+1, \frac{\alpha b_2}{a_2}+1,\frac{b_1}{b_2},\frac{a_1}{a_2}\right)
        \end{align*}
    \end{enumerate} 
    \item Assume that $G_M$ is a well-defined invariant of labeled matroids that satisfies the given recursion and assume $a_2,b_2\neq 0$. We will only show that it is possible to choose $\alpha$ according to the claim as then for $\beta$ we can argue as follows:  Consider the matroid invariant $G_M^*:=G_{M^*}$ that to a matroid $M$ assigns whatever $G_M$ would assign to the dual matroid. Since for $e\in E$ and any matroid $M$ on $E$ we have $(M^*\delete e)=(M\contract e)^*$ and $(M^*\contract e)=(M\delete e)^*$ (whenever these operations are well-defined), we see that $G_M^*$ satisfies the following recursion for all matroids $M$ on $E\neq \emptyset$ and $e\in E$: \begin{align*}
        G_\emptyset^*&=G_\emptyset=\gamma \\
        G_M^*&=\begin{cases}
                (b_1t_e+b_2)G_{M\delete e}^*+(a_1t_e+a_2)G_{M\contract e}^* &\text{ if } e \text{ is neither a loop nor a coloop in }M\\
                (d_1t_e+d_2)G_{M\delete e}^* &\text{ if } e  \text{ is a loop in }  M \\
                (c_1t_e+c_2)G_{M \contract e}^* &\text{ if } e  \text{ is a coloop in }  M. 
        \end{cases}
    \end{align*} 
    In particular we observe that choosing appropriate $\beta$ for $G_M$ is the same as choosing appropriate $\alpha$ for $G_M^*$. \\
    To show that such $\alpha$ exists for $G_M$ we will consider
    two cases, namely whether $b_1$ is zero or not. \\
    \underline{Case 1: $b_1\neq 0$} \\
    First consider the matroid $M=U_{1,3}$ with ground set $E=\{0,1,2\}$. Since the set of bases of this matroid is preserved under any permutation of $E$, any deletion-contraction invariant for this $M$ must return a polynomial in $R[t_0,t_1,t_2]$ that is symmetric in the three variables. We compute $G_M$ recursively by first removing $0$, then 1 and finally 2 from the ground set.  This yields \begin{align*}
        G_{U_{1,3}}={}&(a_1t_0+a_2)G_{U_{1,2}}+(b_1t_0+b_2)G_{U_{0,2}} \\
        ={}&(a_1t_0+a_2)((a_1t_1+a_2)G_{U_{1,1}}+(b_1t_1+b_2)G_{U_{0,1}})+(b_1t_0+b_2)(c_1t_1+c_2)G_{U_{0,1}} \\
        ={}&(a_1t_0+a_2)((a_1t_1+a_2)(d_1t_2+d_2)+(b_1t_1+b_2)(c_1t_2+c_2))+\\&+(b_1t_0+b_2)(c_1t_1+c_2)(c_1t_2+c_2) \\={}&\left({a}_{1}{b}_{2}{c}_{2}+{b}_{1}{c}_{2}^{2}+{a}_{1}{a}_{2}{d}_{2}\right){t}_{0}+\\&+\left({a}_{2}{b}_{1}{c}_{2}+{b}_{2}{c}_{1}{c}_{2}+{a}_{1}{a}_{2}{d}_{2}\right){t}_{1}\\&+\left({a}_{2}{b}_{2}{c}_{1}+{b}_{2}{c}_{1}{c}_{2
     }+{a}_{2}^{2}{d}_{1}\right){t}_{2}+\\&+{a}_{2}{b}_{2}{c}_{2}+{b}_{2}{c}_{2}^{2}+{a}_{2}^{2}{d}_{2}+\text{h.o.t},
    \end{align*}
    where h.o.t stands for the remaining terms which are of degree at least 2 in the variables $t_0,t_1,t_2$
    By the symmetry of the matroid, if we chose to eliminate the elements 0,1,2 in any other order instead, we would get the same expression but with permuted roles of the variables $t_0,t_1,t_2$. In particular the coefficients in front of the linear terms in these variables need to be all equal for $G_M$ to be well-defined. Comparing the $t_0$ and the $t_1$ coefficient, we see \begin{align}
    \label{equation U_1,3 linear}
        {a}_{1}{b}_{2}{c}_{2}+{b}_{1}{c}_{2}^{2}={a}_{2}{b}_{1}{c}_{2}+{b}_{2}{c}_{1}{c}_{2}.
    \end{align}
    As long as $c_2\neq 0$, this allows us to define \begin{align*}
        \alpha:= \frac{c_1-a_1}{b_1}=\frac{c_2-a_2}{b_2}, 
    \end{align*}
    which gives \begin{align*}
        c_1&=b_1\alpha+a_1,\\
        c_2&=b_2\alpha+a_2
    \end{align*}
    as desired. \\
    This proof seems to only work for $c_2\neq 0$ as we had to divide by $c_2$ on our way. However, we assumed that at least one of the parameters $c_1,c_2,d_1,d_2$ is non-zero (such that $G_M$ is not just the trivial invariant that returns 0 on all non-empty matroids). By changing the matroid from $U_{1,3}$ to $U_{2,3}$ (again on $E=\{0,1,2\}$) and by also considering degree two terms instead of linear terms, we can conduct the same proof but change by which of the four parameters $c_1,c_2,d_1,d_2$ we need to divide. The terms to consider are recorded in the following table, where we always assume that $G_M$ is build from $M$ on $E=\{0,1,2\}$ recursively by first removing 0, then 1 and lastly 2.\\
\begin{table}[h]
\begin{tabular}{@{}|l|l|l|@{}}
\toprule
            & $M$       & monomials                      \\ \midrule
$c_1\neq 0$ & $U_{1,3}$ & $t_0t_2\leftrightarrow t_1t_2$ \\
$c_2\neq 0$ & $U_{1,3}$ & $t_0\leftrightarrow t_1$       \\
$d_1\neq 0$ & $U_{2,3}$ & $t_0t_1\leftrightarrow t_0t_2$ \\
$d_2\neq 0$ & $U_{2,3}$ & $t_1\leftrightarrow t_2$       \\ \bottomrule
\end{tabular}
\end{table} \\
\underline{Case 2: $b_1=0$} \\
In this case we can choose $\alpha=\frac{c_2-a_2}{b_2}$, which will clearly force $c_2=b_2\alpha+a_2$. Since $b_1=0$ we now need to show that also $c_1=a_1=b_1\alpha+a_1$. If one of the parameters $c_1,c_2$ is non-zero, the same trick as in the first case works: Assuming $c_2\neq 0$, equation (\ref{equation U_1,3 linear}) gives $a_1=c_1$ as desired, and assuming $c_1\neq 0$ we can look at the $t_0t_2$ and $t_1t_2$ term in $U_{2,3}$ to conclude. We will therefore assume now $c_1=c_2=0$. \\
Next we consider $M=U_{1,2}$ on the groundset $E=\{0,1\}$. We recursively compute \begin{align*}
    G_M=(a_1t_0+a_2)(d_1t_1+d_2)+(b_1t_0+b_2)(c_1t_1+c_2)=a_1d_1t_0t_1+a_1d_2t_0+a_2d_1t_1+a_2d_2.
\end{align*}
Again this needs to be a symmetric polynomial, so we conclude $a_1d_2=a_2d_1$. \\
Next we compare the degree two terms of $U_{2,3}$ on $E=\{0,1,2\}$. The coefficients of $t_0t_2$ and $t_1t_2$ after substituting $b_1=c_1=c_2=0$ reveal that $a_1d_2d_1=a_2d_1^2+b_2a_1d_1$. Since we already know $a_1d_2=a_2d_1$, this implies $b_2a_1d_1=0$. In particular if we have $d_1\neq0$, we can conclude $a_1=0=c_1$ as desired. We therefore now assume that also $d_1=0$ and hence for sure $d_2\neq 0$ as otherwise all four parameters $c_1,c_2,d_1,d_2$ would be zero, contradicting our assumption. \\
To conclude we finally consider the matroid $U_{2,3}$ on $E=\{0,1,2\}$ and look at the coefficients in front of the monomials $t_0$ and $t_1$. After substituting $c_1=c_2=d_1=b_1=0$ all terms vanish except for one and we get $b_2a_1d_2=0$, which after noting $b_2,d_2\neq 0$ gives again $a_1=0=c_1$.
\end{enumerate}
\end{proof}

\section{Evaluations of equivariant Tutte polynomial}
\label{sec:evaluations}
In this section we study specializations of the equivariant Tutte polynomial via its evaluations. There are many known combinatorial interpretations of evaluations of the usual Tutte polynomial, some of them are listed in the second column of table \ref{tab:eval_of_tutte}. The first aim of this section is to find analogues for the equivariant Tutte polynomial. Another specialization of Tutte polynomial is the reduced characteristic polynomial. We propose an analogue definition of an equivariant reduced characteristic polynomial which is motivated by  connections between the usual reduced characteristic polynomial and the non-equivariant push-forward of the tautological class $c_{\crk(M)}([Q_M])$, first observed in \cite{Huh}. In the last section we investigate how much information we lose by evaluating the equivariant Tutte polynomial.

\subsection{Combinatorial interpretations}

We have already seen that setting $t_e=0$ for all $e\in E$ results in obtaining the usual Tutte polynomial. In other words, the constant term with respect to the $t_e$-variables is the usual Tutte polynomial. This motivates us to study other coefficients in the $t_e$-variables.

\begin{notation}
    For $S\subset E$ we denote $t_S := \prod_{e\in S}t_e$.
\end{notation}

We start with special evaluations for $r$ and $s$. The following result will allow us to transfer the known results in table \ref{tab:eval_of_tutte} to the equivariant Tutte polynomial.

\begin{proposition}\label{evaluating_r_s}
    Let $A$ be a subset of $E$. The coefficient of $t_{A}$ in $\widehat{T}_M(x,y,1,0)$ is $$(y-1)^{\nl_M(A)}T_{M\contract A}(x,y)$$ and the coefficient of $t_A$ in $\widehat{T}_M(x,y,0,1)$ is $$(x-1)^{\rank(M) - \rank_M(E\delete A)}T_{M\delete A}(x,y).$$
\end{proposition}
\begin{proof}
    First we prove the statement about $\widehat{T}_M(x,y,1,0)$. For any $S\subset E$ we know $\rank_{M\delete A}(S\delete A) = \rank_{M}(S\delete A)$, $\rank_{M\contract A}(S\delete A) = \rank_M (S\cup A) - \rank_M(A)$. In particular, if $S\subset E\delete A$, $\rank_{M\delete A}(S) = \rank_M(S)$, and if $A\subset S$, $\rank_{M\contract A}(S\delete A) = \rank_M(S) - \rank_M(A)$.
    We will use the closed form of the equivariant Tutte polynomial. The term $t_{A}$ can appear only in those summands, for which $A\subset S$. Thus the desired coefficient can be expressed as
    $$
        \sum_{A\subset S\subset E} (x-1)^{\rank(M)-\rank_M(S)} (y-1)^{\nullity_M(S)} =
    $$
    $$
        =\sum_{A\subset S\subset E} (x-1)^{\rank(M) - \rank_M(A) + \rank_M(A) - \rank_M(S)} (y-1)^{|S|- |A| + \rank_M(A)-\rank_M(S) + |A| - \rank_M(A)} =
    $$
    $$
         =
         (y-1)^{|A|-\rank_M(A)}\sum_{A\subset S\subset E} (x-1)^{\rank(M\contract A)-\rank_{M\contract A}(S\delete A)}(y-1)^{|S\setminus A| - \rank_{M}(S\delete A)} 
         =
    $$
    $$
         =
         (y-1)^{|A|-\rank_M(A)}\sum_{S\subset E\setminus A} (x-1)^{\rank(M\contract A)-\rank_{M\contract A}(S)}(y-1)^{|S| - \rank_{M\contract A}(S)}  = (y-1)^{|A|-\rank_M(A)}T_{M\contract A}(x,y).
    $$
    Now we will prove the second part of the proposition by passing to dual matroid. Recall $\rank_{M^*}(A) = |A| + \rank_M(E\setminus A) - \rank(M)$.
    \begin{align*}
        \widehat{T}_M(x,y,0,1) &= \widehat{T}_{M^*}(y,x,1,0) = (x-1)^{|A| - \rank_{M^*}A}T_{M^*\contract A}(y,x) =\\
          &= (x-1)^{\rank(M)-\rank_M(E\delete A)}T_{M\delete A}(x,y).\qedhere
    \end{align*}
\end{proof}
We now give a list of some of evaluations of $\widehat{T}_M$ and their combinatorial interpretations.

\begin{theorem}
\label{thm:evaluations}
For a matroid $M$ on a groundset $E$ and a subset $A\subseteq E$ there are combinatorial interpretations of the coefficients of $t_A$ after special evaluations, listed in table \ref{tab:eval_of_tutte}. The first column contains the values of $(x,y,r,s)$ that are substituted, the second column refers to $T_M(x,y)$ and the last column describes the coefficient of the monomial $t_A$ in $\widehat{T}_M(x,y,r,s)$. For the last two lines, we work with a graph, its corresponding graphic matroid and Tutte polynomial.
\begin{table}[h!]
    \centering
    \renewcommand{\arraystretch}{1.8}
    \setlength{\tabcolsep}{.65em}
    \begin{tabular}{p{1.8cm}p{3.5cm}p{9cm}}
         Evaluation & Classical result & Combinatorial interpretation of the coefficient of $t_A$ \\
         \hline
         (1,1,1,0) & Number of bases & Number of bases containing $A$\\
         (2,2,1,0) & $2^{|E|}$ & $2^{|E\delete A|}$ \\
         (2,1,1,0) & Number of independent sets & Number of independent sets containing $A$ \\
         (1,2,0,1) & Number of spanning sets & Number of spanning sets disjoint from $A$ \\
         (2,0,1,0) & Number of acyclic orientations & For $A$ independent: number of orientations of $E\delete A$, so that the orientation of all edges is acyclic no matter how $A$ is oriented.\\
         (0,2,1,0) & Number of strongly connected orientations & Number of strongly connected orientations if we allow edges from $A$ to be directed in both ways\\
    \end{tabular}
    \caption{Special evaluations of the equivariant Tutte polynomial}
    \label{tab:eval_of_tutte}
\end{table}

\end{theorem}
\begin{proof}
    All the conclusions follow from the classical results and \ref{evaluating_r_s}. For an illustration we give a detailed proof of the case of evaluating in (1,1,1,0).
     By \ref{evaluating_r_s} we know the coefficient of $t_A$ is $(1-1)^{\nl_M(A)}T_{M\contract A}(1,1)$. Clearly this is 0 when $A$ is not independent, i.e. if there is no basis containing $A$. If $A$ is independent, this is $T_{M\contract A}(1,1)$. By the corresponding classical result this is the number of bases of $M\contract A$. Since $A$ is independent, unifying a basis of $M\contract A$ with $A$ gives a basis of $M$. On the other hand for all bases $B$ of $M$ containing $A$, $B\setminus A$ is a basis of $M\contract A$.
\end{proof}
\begin{remark}
    Setting $A=\emptyset$ in the previous theorem gives the classical results, so in this case the last two columns of \ref{tab:eval_of_tutte} agree. Since $t_\emptyset = 1$, we can find the classical result as the free term of the evaluated polynomial.
\end{remark}
We now continue Example \ref{Example:cycle} of a graphic matroid coming from just one cycle of length $n+1$.
\begin{example}
\label{Example:cycle-continued}
    Let $M=U_{n,n+1}$, we computed in \ref{Example:cycle}, that \begin{align*}
        \widehat{T}_M(x,y,r,s)=\sum_{k=0}^{n+1}\left(\frac{(xs-s+r)^{k}x^{n+1-k}-r^{k}}{x-1}+(y-1)r^{k}\right)\elem_k(\{t_e\}_{e\in E}).
    \end{align*}
    Evaluating at $(r,s)=(1,0)$ this yields \begin{align*}
        \widehat{T}_M(x,y,1,0)&=\sum_{k=0}^{n+1}\left(\frac{x^{n+1-k}-1}{x-1}+y-1\right)\elem_k(\{t_e\}_{e\in E}) \\&=
        \sum_{k=0}^{n}\left(y+\sum_{i=1}^{n-k}x^{i}\right)\elem_k(\{t_e\}_{e\in E})+(y-1)\elem_{n+1}(\{t_e\}_{e\in E}).
    \end{align*}
    This means that the coefficient of $t_A$ for some $k$-element subset $A\subsetneq E$ and $k\neq n+1$ is precisely $y+\sum_{i=1}^{n-k}x^{i}$, which is the classical Tutte polynomial of the $n+1-k$ cycle $U_{n-k,n+1-k}$. Indeed this is precisely the matroid that we obtain from $M$ by contracting any $k$-element subset, so \ref{evaluating_r_s} confirms our computation. Even for $A=E$ we obtain the result from \ref{evaluating_r_s} by noting that in this case $(y-1)^{|A|-\rank_M(A)}T_{M\contract A}=(y-1)$. 
\end{example}

We have been evaluating only in $x,y,r$ and $s$ so far. By evaluating some of the $t_e$-variables we can recover the equivariant Tutte polynomial for smaller matroids.
\begin{proposition}\label{evaluating_t_e}
    Let $A\subset E$. Then
    \begin{align*}
        \left.\widehat{T}_M(x,y,r,s)\right|_{t_e = -\frac{1}{s}, e\in A} &= \left(1-\frac{r}{s}\right)^{|A|}(y-1)^{\nullity_M(A)}\widehat{T}_{M\contract A}(x,y,r,s),\\
        \left.\widehat{T}_M(x,y,r,s)\right|_{t_e = -\frac{1}{r}, e\in A} &= \left(1-\frac{s}{r}\right)^{|A|}(x-1)^{\rank(M) - \rank(E\delete A)}\widehat{T}_{M\delete A}(x,y,r,s).
    \end{align*}
\end{proposition}
\begin{proof}
    We proceed similar to the proof of \ref{evaluating_r_s}.
    $$
        \widehat{T}_M(x,y,r,s)|_{t_e = -\frac{1}{s}, e\in A} = 
    $$
    $$
        =\sum_{A\subset S\subset E} (x-1)^{\rank(M)-\rank_M(S)}(y-1)^{\nullity_M(S)} \prod_{e\in A} \left(1-\frac{r}{s}\right) \prod_{e\in S\delete A} (1+r t_e) \prod_{e\notin S} (1+s t_e) +
    $$
    $$
        +\sum_{A\not\subset S\subset E} (x-1)^{\rank(M)-\rank_M(S)}(y-1)^{\nullity_M(S)} \prod_{e\in S} (1+r t_e) \cdot 0 =
    $$
    $$
       = \left(1-\frac{r}{s}\right)^{|A|}\sum_{A\subset S\subset E} (x-1)^{\rank(M)-\rank_M(A) + \rank_M(A)-\rank_M(S)}(y-1)^{\nullity_M(S)-\nullity_M(A) + \nullity_M(A)}\cdot
     $$
     $$
        \cdot\prod_{e\in S\delete A} (1+r t_e) \prod_{e\notin S} (1+s t_e) =
    $$
    $$
        = \left(1-\frac{r}{s}\right)^{|A|}(y-1)^{\nullity_M(A)}\sum_{S\subset E\delete A} (x-1)^{\rank(M\contract A)-\rank_{M\contract A}(S)}(y-1)^{\nullity_{M\contract A}(S)} \prod_{e\in S\delete A} (1+r t_e) \prod_{e\notin S} (1+s t_e) =
    $$
    $$
        =\left(1-\frac{r}{s}\right)^{|A|}(y-1)^{\nullity_M(A)}\widehat{T}_{M\contract A}(x,y,r,s).
    $$
    For the second statement, we again pass to the dual matroid.
    $$
        \widehat{T}_M(x,y,r,s)|_{t_e = -\frac{1}{r}, e\in A} = \widehat{T}_{M^*}(y,x,s,r)|_{t_e = -\frac{1}{r}, e\in A} =
    $$
    $$
        =\left(1-\frac{s}{r}\right)^{|A|}(x-1)^{\rank(M) - \rank(E\delete A)}\widehat{T}_{M^*\contract A}(y,x,s,r) = 
    $$
    $$
        =\left(1-\frac{s}{r}\right)^{|A|}(x-1)^{\rank(M) - \rank(E\delete A)}\widehat{T}_{M^\delete A}(x,y,r,s).
    $$
\end{proof}

\subsection{Equivariant characteristic polynomial}
Another polynomial associated to a matroid is its reduced characteristic polynomial \begin{align*}
    \chi_M(q):= (-1)^{\rank(M)}\frac{T_M(1-q,0)}{q-1}\in \Z[q].
\end{align*}
In the case of a graphic matroid coming from some graph $G$, the above notion agrees with the reduced chromatic polynomial of the graph $G$ up to a factor $q^{c(G)}$ where $c(G)$ denotes the number of connected components of $G$ and therefore is not a matroid invariant. The latter polynomial, when multiplied by $q-1$, evaluates for any natural number $q$ to the number of valid vertex colorings of the graph $G$.

Just as we did to define the equivariant Tutte polynomial, we will recall a geometric interpretation of the reduced characteristic polynomial and lift it to an equivariant setting. The geometric interpretation was first observed by June Huh in \cite{Huh}. However, it can also be seen from \ref{tautological pushforward}. Indeed non-equivariantly we have \begin{align*}
    \int_{X_E}c([\Sbundle_{U_{n,E}}^\vee],x)c([\Qbundle_{U_{1,E}}],y)c_{\corank(M)}([\Qbundle_M])\left|_{y=-1}\right.=\chi_M(x)
\end{align*}
Hence to define an equivariant analog of the reduced characteristic polynomial, we will pushforward the top Chern class of the $\Qbundle_M$ bundle. Since we already know how to pushforward the whole graded Chern class already, this is straightforward.
\begin{proposition}[Pushforward of the top Chern class of the quotient bundle]
    $$
    \pi^n_*\Big(c_{\corank(M)}([\Qbundle_M])\Big) = P_M(\alpha,\beta)
    $$
    holds in $H_T(\P^n\times\P^n)$, where $P_M\in\Z[E][\alpha,\beta]$ is given by
   $$
        P_M(\alpha,\beta) = \frac{1}{\alpha+\beta} \sum_{S\subset{E}} \alpha^{\corank(M)-\nullity_M(S)}(-\beta)^{\nullity_M(S)}\cdot\prod_{e\in S} (\alpha+t_e)\prod_{e\notin S}(\beta-t_e).
   $$
\end{proposition}
\begin{proof}
    This follows from \ref{thrm:pushforward-graded-chern} by looking at the coefficient of $z^0w^{\corank(M)}$.
\end{proof}
Again we want to reverse the grading in the variables $\alpha,\beta$ for the reasons described in Section \ref{chapter 4}. Then we do the same substitutions as in the non-equivariant case, i.e. we set $\beta=-1$ and read $\alpha$ as a formal variable. This results in the following definition.
\begin{definition}
    Let $M$ be a matroid on a groudset $E$. We define the equivariant reduced characteristic polynomial $\widehat{\chi}_M(q)\in \Z[E][q]$ of $M$ as follows:
    $$
        \widehat{\chi}_M(q) = \frac{(-1)^{\rank(M)}}{q-1}\cdot \widehat{T}_M(1-q, 0, q, 1) = \frac{1}{q-1} \cdot \sum_{S\subset E} q^{\rank(M)-\rank_M(S)}(-1)^{|S|}\cdot \prod_{e\notin S} (1+t_e) \prod_{e\in S} (1+qt_e).
    $$
\end{definition}

Indeed we have the following relation between the equivariant reduced characteristic polynomial and the pushforward of the top Chern class of the quotient bundle of $M$:

\begin{proposition}[Relating equivariant reduced characteristic polynomial and the pushforward]
    Let $P_M$ be as in the previous proposition. Then
    $$
        P_M\left(\frac{1}{q}, -1\right) = \frac{1}{(-q)^{|E|-1}} \cdot \widehat{\chi}_M(q)
    $$
\end{proposition}
\begin{proof}
    $$
        P_M\left(\frac{1}{q}, -1\right) = \frac{1}{\frac{1}{q}-1}\sum_{S\subset E}\frac{1}{q^{\corank(M)-\nullity_M(S)}}\cdot 1^{\nullity_M(S)} \cdot \prod_{e\in S} \left(\frac{1}{q} + t_e\right) \prod_{e\notin S} (-1-t_e)=
    $$
    $$
        = \frac{-q}{q-1}\sum_{S\subset E} q^{\rank(M)-\rank_M(S)}q^{-|E|}q^{|S|-|S|}\cdot (-1)^{|E\setminus S|}\cdot \prod_{e\in S} (1+qt_e) \prod_{e\notin S}(1+t_e)=
    $$
    $$
        = \frac{1}{(-q)^{|E|-1}}\cdot \frac{1}{q-1}\sum_{S\subset E}\cdot (-1)^{|S|}\cdot \prod_{e\notin S}(1+t_e)\prod_{e\in S} (1+qt_e) =
        \frac{1}{(-q)^{|E|-1}} \cdot \widehat{\chi}_M(q)
    $$
\end{proof}
Notice that setting $t_e=0$ for all $e\in E$, 
we recover the usual characteristic polynomial. More generally, we have the following result.
\begin{proposition}[Combinatorial interpretation of the equivariant reduced characteristic polynomial]
\label{usual_char_poly}
    Let $A\subset E$. Then setting $t_e=-1$ for $e\in A$ and $t_e=0$ for $e\notin A$ in the  equivariant reduced characteristic polynomial $\widehat{\chi}_M(q)$ gives
    $$
        (1-q)^{|A|}(-1)^{|A|}\chi_{M\contract A}(q).
    $$
\end{proposition}
\begin{proof}
    We will use Proposition \ref{evaluating_t_e} to show that substituting $t_e=-1$ for $e\in A$ and $t_e=0$ for $e\notin A$ yields the desired result.
    $$
        \widehat{\chi}_M(q)|_{t_e=-1,e\in A} = \frac{(-1)^{\rank(M)}}{q-1}\cdot \widehat{T}_M(1-q, 0, q, 1)|_{t_e=-1,e\in A} = 
    $$
    $$
        =\frac{(-1)^{\rank(M)}}{q-1} \cdot \left(1-q \right)^{|A|}(0-1)^{\nullity_M(A)}\widehat{T}_{M\contract A}(1-q, 0, q, 1)=
    $$
    $$
        =\frac{(-1)^{\rank(M\contract A)}}{q-1} \cdot \left(q-1 \right)^{|A|}\widehat{T}_{M\contract A}(1-q, 0, q, 1) = (1-q)^{|A|}(-1)^{|A|}\widehat{\chi}_{M\contract A}(q).
    $$
    By setting the remaining $t_e$ to zero we obtain the usual reduced characteristic polynomial.
\end{proof}
An easy corollary to the valuativity of $\widehat{T}_M$ (Proposition \ref{valuativity}) is that the equivariant reduced characteristic polynomial is also valuative. Alternatively this also follows immediately from noting that $c_{\rank{M}}([\Q_M])$ is valuative by \cite{tautological-classes}, Proposition 5.6 and the pushforward map is additive.
\begin{corollary}
The assignment $M\mapsto \widehat{\chi}_M$ that assigns to a matroid $M$ its equivariant reduced characteristic polynomial is valuative.
\end{corollary}

Since we already computed $\widehat{T}_M$ for $M=U_{n,n+1}$ in Examples \ref{Example:cycle} and \ref{Example:cycle-continued}, we can now easily obtain its equivariant reduced characteristic polynomial.
\begin{example}
    Let $M=U_{n,n+1}$, recall from \ref{Example:cycle} that we have \begin{align*}
        \widehat{T}_M(x,y,r,s)=\sum_{k=0}^{n+1}\left(\frac{(xs-s+r)^{k}x^{n+1-k}-r^{k}}{x-1}+(y-1)r^{k}\right)\elem_k(\{t_e\}_{e\in E}).
    \end{align*}
    Note that the substitution $(x,y,r,s)=(1-q,0,q,1)$ makes the term $(xs-s+r)^k$ vanish for $k>0$, so we split the sum to get \begin{align*}
        \widehat{\chi}_M(q)=\frac{(-1)^{n}}{q-1}\left(\frac{(1-q)^{n+1}-1}{-q}-1+\sum_{k=1}^{n+1}(q^{k-1}-q^k)\elem_k(\{t_e\}_{e\in E})\right)
    \end{align*}
    In fact if we fix some $k$-element subset $A$ and set $t_e=-1$ for $e\in A$ and $t_e=0$ else, only the terms $t_S$ for $S\subseteq A$ are non-zero. We hence obtain contributions from the first $k$ summands only. Hereby the summand corresponding to a cardinality $l\leq k$ subset of $A$ contributes once for each $l$-element subset of $A$, i.e. with factor $(-1)^{l}\binom{k}{l}$. We hence get  \begin{align*}
        \widehat{\chi}_M(q)|_{t_e=-1,e\in A,t_e=0, e\notin A}=\frac{(-1)^{n}}{q-1}\left(\frac{(1-q)^{n+1}-1}{-q}-1+\sum_{l=1}^{k}\binom{k}{l}(q^{l-1}-q^l)(-1)^l\right).
    \end{align*}
    Careful manipulations of this expression indeed yield \begin{align*}
        \widehat{\chi}_M(q)|_{t_e=-1,e\in A,t_e=0, e\notin A}=(1-q)^k(-1)^k\frac{(-1)^{n-k}}{q-1}\sum_{i=1}^{n-k}(1-q)^i= (1-q)^k(-1)^k\chi_{U_{n-k,n+1-k}}(q)
    \end{align*}
    and hence confirms the connection to the usual reduced characteristic polynomial of an $n+1-k$-cycle, which was predicted by \ref{usual_char_poly}.
\end{example}

\subsection{Uniqueness of evaluation for matroids}
We have seen that after evaluating at $(1,1,1,0)$ in $\widehat{T}_M$, $t_A$ has non-zero coefficient if and only if $A$ is independent. In particular, we can recover $M$ from $\widehat{T}_M$. It turns out there are not many evaluations destroying this uniqueness as long as we keep the $t_e$-variables untouched.
\begin{proposition}
    Let $R$ be an integral domain and pick $x,y,r,s\in R$. Assume there exist two different matroids $M_1,M_2$ on the same groundset $E$ such that  $\widehat{T}_{M_1}(x,y,r,s) = \widehat{T}_{M_2}(x,y,r,s)$ holds in $R[E]$, then either $r = s$ or $(x-1)(y-1) = 1$.
\end{proposition}
\begin{proof}
For brevity we will write $\widehat{T}_{M}$ instead of $\widehat{T}_{M}(x,y,r,s)$ for any matroid $M$. Also note that since we can pass to the fraction field of $R$, we may assume that $R$ is a field.

We will proceed by a contradiction. Suppose that $M_1$ and $M_2$ are different matroids, $r\neq s$, $(x-1)(y-1)\neq 1$ but $\widehat{T}_{M_1}=\widehat{T}_{M_2}$. Furthermore, suppose that the cardinality of $E$ is the smallest possible among such examples of $M_1$ and $M_2$.
    We consider three cases:
    \begin{enumerate}
        \item There exists $\hat{e}\in E$ such that it is general in both $M_1$ and $M_2$. Then
        $$
            (1+s t_{\hat{e}})\widehat{T}_{M_1\delete \hat{e}} + (1+r t_{\hat{e}})\widehat{T}_{M_1\contract \hat{e}} = \widehat{T}_{M_1} = \widehat{T}_{M_2} = (1+s t_{\hat{e}})\widehat{T}_{M_2\delete \hat{e}} + (1+r t_{\hat{e}})\widehat{T}_{M_2\contract \hat{e}}.
        $$
        These polynomials can be regarded as one variable polynomials in $t_{\hat{e}}$ with coefficients in $R[t_e\mid e \in E\setminus \{\hat{e}\}]$. Their equality yields the equality of their corresponding coefficients. Comparing the constant term we get
        \begin{equation}\label{comp}
            \widehat{T}_{M_1\delete \hat{e}} - \widehat{T}_{M_2\delete \hat{e}} = \widehat{T}_{M_2\contract \hat{e}} - \widehat{T}_{M_1\contract \hat{e}}
        \end{equation}
        and similarly for the linear terms we get
        \begin{equation} \label{comp2}
                       s \widehat{T}_{M_1\delete \hat{e}} - s\widehat{T}_{M_2\delete \hat{e}} = r \widehat{T}_{M_2\contract \hat{e}} - r \widehat{T}_{M_1\contract \hat{e}}.
        \end{equation}
        Denoting $c$ the constant term, that is equal to both sides of \eqref{comp}, we can rewrite \eqref{comp2} as $(s-r) c = 0$. Since $r \neq s$, we obtain $c = 0$. By minimality of $|E|$ this can be true only if $M_1\setminus \hat{e}=M_2 \setminus \hat{e}$ and $M_1\contract \hat{e}=M_2\contract \hat{e}$. But this implies that $M_1$ and $M_2$ are the same matroid, a contradiction \footnote{Notice that we really use the fact that $M_1\setminus \hat{e}=M_2 \setminus \hat{e}$ and $M_1\contract \hat{e}=M_2\contract \hat{e}$ are equalities of labeled matroids. In general for a matroid $M$ and an element $e$ from its groundset, it is not true that knowing $M\contract e$ and $M\delete e$ up to isomorphism uniquely determines $M$ up to isomorphism.}.

        \item There exists $\hat{e}\in E$ such that it is general in one of $M_1$, $M_2$ but it is loop or coloop in the other one. Since the situation is symmetric with respect to $M_1$ and $M_2$, let us presume $\hat{e}$ is general in $M_1$. Similarly, passing to dual matroids changes $\hat e$ being a coloop in $M_2$ to it being a loop, so let us presume that $\hat e$ is a loop in $M_2$.

        Then we have
        $$
             (1+s t_{\hat{e}})\widehat{T}_{M_1\delete \hat{e} } + (1+r t_{\hat{e}})\widehat{T}_{M_1\contract \hat{e} } = \widehat{T}_{M_1} = \widehat{T}_{M_2} = (y + (yr - r + s)t_{\hat{e}})\widehat{T}_{M_2\delete \hat{e} }.
        $$
        Comparing the coefficients with respect to the variable $t_{\hat{e}}$, we get
        \begin{align}\label{eq:comp_const_b}
            \widehat{T}_{M_1\delete \hat{e}} + \widehat{T}_{M_1\contract \hat{e}} &= y\widehat{T}_{M_2\delete \hat{e}},\\
            \label{eq:comp_lin_b}
            s \widehat{T}_{M_1\delete \hat{e}} + r \widehat{T}_{M_1\contract \hat{e}} &= (yr - r + s)\widehat{T}_{M_2\delete \hat{e}}.
        \end{align}
        We can substitute $y\widehat{T}_{M_2\delete \hat{e}}$ from the first expression into the second, obtaining
        $$
            (s - r)  \widehat{T}_{M_1\delete \hat{e}} = (s - r) \widehat{T}_{M_2\delete \hat{e}}.
        $$
        Dividing by $s - r$ we obtain $ \widehat{T}_{M_1\delete \hat{e}} = \widehat{T}_{M_2\delete \hat{e}}$. Thanks to the minimal choice of $M_1$ and $M_2$, this yields $M_1\delete \hat{e} = M_2\delete \hat{e}$.
        
        Now, if there were a general element in the matroid $M_1\delete \hat{e} = M_2\delete \hat{e}$, it would remain general in both $M_1$ and $M_2$. This would allow us to reduce to the case (a). Consequently, if suffices for us to treat just the case where no element of $M_1\delete \hat{e} = M_2\delete \hat{e}$ is general.

        First let us assume $M_1$ has no loops nor coloops. As deletion does not create any new loops, $M_1\delete \hat{e}$ is a collection of coloops. So $M_1 = U_{n,n+1}$. Further, we deduce that $M_2$ consists of $n$ coloops and one loop.
         If $s\neq 0$, we can use Proposition~\ref{evaluating_t_e}. We choose some $e'\in E\setminus \hat{e}$ and set $t_{e'}$ to $-\frac{1}{s}$.  
        The element $e'$ is a coloop in $M_2$ and a general element in $M_1$.
        \begin{align*}
            \left(1-\frac{r}{s}\right) \widehat{T}_{M_1\contract e'} = \widehat{T}_{M_1}|_{t_{e'}=-\frac{1}{s}} &= \widehat{T}_{M_2}|_{t_{e'}=-\frac{1}{s}} = \left(1-\frac{r}{s}\right) \widehat{T}_{M_2\contract e'}\\
            \widehat{T}_{M_1\contract e'} &= \widehat{T}_{M_2\contract e'}
        \end{align*}
        Since $M_1\contract e' = U_{n-1,n}$ and $M_2\contract e'$ is a collection of $n-1$ coloops and one loop, we get a contradiction with minimality of  $E$. So $s=0$.
        
        From $\eqref{eq:comp_const_b}$ we get
        \begin{align*}
            \widehat{T}_{U_{n-1,n}} &= (y-1)\widehat{T}_{U_{n,n}}
        \end{align*}
         We can expand those expressions by the definition of the equivariant Tutte polynomial. Since $\nullity_M(S)=0$ and $\rank_M(S) = |S|$ for independent $S$, it will simplify to:
        \begin{align*}
             (x-1)^{0}(y-1)^1\prod_{e\in E}(1+r t_e) + \sum_{S\subsetneq E}&(x-1)^{n-1 - |S|} (y-1)^0 \prod_{e\in S}(1+r t_e) \prod_{e\notin S} (1+0\cdot t_e)=\\
             =(y-1)\sum_{S\subset E} &(x-1)^{n - |S|} (y-1)^0\prod_{e\in S}(1+r t_e) \prod_{e\notin S} (1+0\cdot t_e),\\
             \sum_{S\subsetneq E}(x-1)^{n-1 - |S|} \prod_{e\in S}(1+r t_e)&=(y-1)(x-1)\sum_{S\subsetneq E} (x-1)^{n-1 - |S|}\prod_{e\in S}(1+r t_e)
        \end{align*}
        We get a contradiction with $(x-1)(y-1) \neq 1$, unless $$
            \sum_{S\subsetneq E}(x-1)^{n-1 - |S|} \prod_{e\in S}(1+r t_e) = 0.
        $$
        Then the coefficients of $t_A$ for all $A\subset E$ must be zero. We choose $A$ such that $|A|=n-1$. Then coefficient of $t_A$ is simply $r^{n-1}$, since $t_A$ can appear only if $S=A$. Since $s=0$, $r\neq 0$ and this coefficient is non-zero.

        Now we assume there exists $e'\in E$ such that it is a loop in $M_1$. It has to be a loop in $M_2$ too. Then 
        $$
            (y+(yr + s - r)t_e)\widehat{T}_{M_1\delete e'} = \widehat{T}_{M_1} = \widehat{T}_{M_2} = (y+(yr + s - r)t_e)\widehat{T}_{M_2\delete e'}.
        $$
        Since $M_2\delete e'$ is collection of loops and coloops and $M_1\delete e'$ is not, certainly $M_1\delete e' \neq M_2\delete e'$. So $y+(yr + s - r)t_e = 0$, else we get a contradiction by the minimality of $E$ again. But this 
        yields $y=0$, $s = r$, a contradiction.

        The case when $e'$ is a coloop can be treated analogously.

        \item Both $M_1$ and $M_2$ are collections of loops and coloops. Let us denote $f_e := [(y-1)(1+r t_{e})+(1+s t_{e})]$ and $g_e := [(x-1)(1+s t_{e})+(1+r t_{e})]$. If $L_1$, resp. $L_2$, is the set of all loops in $M_1$, resp. $M_2$, then we get from Example~\ref{all-loops-coloops} that
        $$\label{product-equation}
        \prod_{e\in L_1} f_e \prod_{e\notin L_1}g_e = \widehat{T}_{M_1} = \widehat{T}_{M_2}=\prod_{e\in L_2} f_e \prod_{e\notin L_2}g_e.$$
        Note that since $R$ is a field, $R[E]$ is a UFD and the $f_e,g_e$ are linear in the $t_e$-variables and therefore irreducible. Hence the factors on the left- and right-hand side of \ref{product-equation} are pairwise associated to each other. As for any $e\in E$, the factors $f_e$ or $g_e$ respectively only use the variable $t_e$ and none of the others, we know exactly which factors are paired up.
         By our assumption the matroids are different, so there exists $\hat{e}\in E$ that is a loop in one of the matroids and a coloop in the other one. Without loss of generality we may assume that $\hat{e}\in L_1\setminus L_2$, then we get that $f_{\hat{e}}$ and $g_{\hat{e}}$ are associated. Being both linear in $t_{\hat e}$, this means they differ only by a constant multiple (with respect to $t_{\hat e}$). Writing them as $f_{\hat{e}} = at_{\hat{e}} + b$ and $g_{\hat{e}} = ct_{\hat{e}} + d$, with $a = (y-1)r + s, b = (y-1) + 1, c = (x-1)s +r$ and $d = (x-1) + 1$, we get the equation $ad = bc,$ i.e.
        \begin{align*}
            y\cdot (xs-s+r) &= x\cdot (yr-r+s)\\
            (x-1)(y-1)(s - r) &= s - r\\
            (x-1)(y-1) &= 1,
        \end{align*}
        and we arrive at the desired contradiction.
    \end{enumerate}
\end{proof}

Let us note that the last proposition heavily relies on the fact $M_1$ and $M_2$ are on the same groundset. Usually, the groundset can be read off from the variables $t_e$. But for example if $x=1$, $r=0$, contraction relation for a coloop $\hat{e}$ from a matroid $M$ simplifies to $\widehat{T}_M = \widehat{T}_{M\contract \hat{e}}$, so the variable $t_{\hat{e}}$ never appears. Analogously, if $y=1$ and $s=0$, $t_{\hat{e}}$ never appears if $\hat{e}$ is a loop. Those evaluations together with $E$ still uniquely determine the matroid, namely the evaluation $(1,1,1,0)$ directly gives the independent sets.


\end{document}